\documentclass{amsproc}

\usepackage{amsmath,amssymb}
\usepackage[figuresright]{rotating}
\usepackage{rotating}
\newtheorem{theorem}{Theorem}[section]
\newtheorem{lemm}[theorem]{Lemma}
\newtheorem{prop}[theorem]{Proposition}
\newtheorem{coro}[theorem]{Corollary}

\theoremstyle{definition}
\newtheorem{defi}[theorem]{Definition}

\theoremstyle{remark}
\newtheorem{remark}[theorem]{Remark}
\numberwithin{equation}{section}

\def\vn{\varepsilon}

\def\ot{\otimes}
\def\b{\overline}

\def\om{\omega}

\def\dim{\hbox{dim}}

\def\a{\alpha}

\def\b{\beta}

\newfont{\df}{eufm10}

\def\ot{\otimes}

\def\dim{\hbox{\rm dim}\,}

\def\ot{\otimes}

\begin{document}

\title[PBW-Bases and Restricted Two-Parameter Quantum Group of Type $F_4$]
{Convex PBW-type Lyndon Bases and Restricted Two-Parameter Quantum Group of Type $F_4$}
\author[Chen]{Xiaoyu Chen$^{\dagger,a,b}$}
\address{$^a$Department of Mathematics, SKLPMMP, East China Normal University,
Min Hang Campus, Dong Chuan Road 500, Shanghai 200241, PR China}
\address{$^b$Department of Mathematics, Ningbo University,
Fenghua Road 818, Ningbo 315211, PR China}
\email{gauss\_1024@126.com}
\thanks{$^\dagger$Xiaoyu Chen,
supported by the NNSF of China (Grant No. 11501546).}
\author[Hu]{Naihong Hu$^{\star,a}$}
\address{$^a$Department of Mathematics, SKLPMMP, East China Normal University,
Min Hang Campus, Dong Chuan Road 500, Shanghai 200241, PR China}
\email{nhhu@math.ecnu.edu.cn}
\thanks{$^\star$Naihong Hu,
supported by the NNSF of China (Grant No. 11271131).}
\author[Wang]{Xiuling Wang $^*$ }
\address{School of Mathematical Sciences and  LPMC, Nankai University,
Tianjin  300071, PR China}\email{xiulingwang@nankai.edu.cn}
\thanks{$^*$Xiuling Wang, corresponding author, supported by the NNSF of China (Grant No. 11271131).}

%    General info
\date{}


%%%%%%%%%%%%%%%%%%%%%%%%%%%%%%%%%%%%%%%%%%%%%%%%%%%%%%%%%%%%%%%%%%%%%%%%
%\footnote{Corresponding author.}%
%%%%%%%%%%%%%%%%%%%%%%%%%%%%%%%%%%%%%%%%%%%%%%%%%%%%%%%%%%%%%%%%%%%%%%%%
\begin{abstract}
We determine  convex PBW-type Lyndon bases for two-parameter quantum
groups $U_{r,s}(F_4)$ with detailed commutation relations.  We
construct a finite-dimensional Hopf algebra $\mathfrak
u_{r,s}(F_4)$, as a quotient of $U_{r,s}(F_4)$ by a Hopf ideal generated by certain central elements, which is pointed, and of a
Drinfel'd double structure under a certain condition.  All of Hopf
isomorphisms of $\mathfrak u_{r,s}(F_4)$ are determined which are important for seeking the possible new pointed objects in low order with $(\ell, 210)\ne 1$. Finally,
necessary and sufficient conditions for $\mathfrak u_{r,s}(F_4)$ to
be a ribbon Hopf algebra are singled out by describing the left and
right integrals.
\end{abstract}

\maketitle
\section{Introduction}
In the past fifteen years, a systematic study of the two-parameter
quantum groups has been going on, see, for instance, \cite{BW1, BW2}
for type $A$, \cite{BGH1, BGH2, BH, HS} for other finite types
except type $F_4$; for a unified definition, see \cite{HP};
\cite{HRZ, HZ1, HZ2, GHZ} for the affine types;  \cite{BW3, HW1,
HW2} for the restricted two-parameter quantum groups for some of
finite types.  In the present paper, we continue to treat with the
same questions for type $F_4$ in the roots of unity case. We give
the explicit description of PBW-type Lyndon bases of two-parameter
quantum group of type $F_4$ and study the restricted two-parameter
quantum group $\mathfrak u_{r,s}(F_4)$. Since, in the cases of
non-simply-laced Dynkin diagrams, much complexity in combinatorics
of Lyndon bases for type $F_4$ causes more inconvenience than those
for types $B$ and $G_2$, our techniques required here are somewhat
more subtle than those in \cite{HW1, HW2}.

\smallskip

We survey some previous related work.  The constructions of convex
PBW-type Lyndon bases in the two-parameter quantum groups have been
given for types $A$ in \cite{BW3}, types $E$ in \cite{BH}, for type
$G_2$ in \cite{HW1}, and for types $B$ in \cite{HW2}. When studying
the convex PBW-type Lyndon bases of the two-parameter quantum groups
and the restricted two-parameter quantum groups, we need use more
complicated combinatorics than those in one-parameter. The
nice bases for the corresponding nonrestricted quantum groups at
generic case are important to the studies of restricted (or small)
quantum groups as finite-dimensional Hopf algebras, even to the
liftings of small quantum groups in the classification work of
finite-dimensional pointed Hopf algebras due to
Andruskiewitsch-Schneider (\cite{AS2}). Good Lyndon words and Lyndon
bases have been described by Lalonde-Ram for semisimple Lie
algebras and their universal enveloping algebras \cite{LR}. 
Combinatorial constructions of quantum PBW-type Lyndon bases depend
not only on a fixed choice of a convex ordering on a positive root
system which is given by the longest element of (finite) Weyl group,
but also on the inserting manner of the $\bold q$-bracketings with
suitable structure constants into good Lyndon words (also see
\cite{K} for types $ABCD$). Lyndon bases have been also applied in
the classification of finite-dimensional pointed Hopf algebras or
Nichols algebras of diagonal type. Leclerc \cite{L} defined the
``co-standard factorization" of a Lyndon word which is different but
equivalent to that in \cite{LR, Ro} (while the case for type $F_4$ wasn't treated in \cite{L}). We make an inductive definition
of quantum root vectors for $F_4$ by the co-standard factorization
of good Lyndon words, which is different from the construction by
the standard factorization of good Lyndon words for types $ABG$.
Like one-parameter quantum groups in \cite{LS}, we provide more explicit commutator
expressions among these quantum root vectors, which are crucial later on.

\smallskip

 The paper is organized as follows.  In Section 2, we begin with the definition of two-parameter quantum group of type $F_4$,
 and give an explicit description of PBW type Lyndon bases. In Section 3, we compute some
important basic commutation relations among quantum root vectors and then derive a unified description of all commutation
relations.
These give rise to central elements of degree
$\ell$ in the case when parameters $r$, $s$ are roots of unity,
which generate a Hopf ideal. The quotient is the restricted
two-parameter quantum group $\mathfrak
u_{r,s}(F_4)$ in Section 4.  In Section 5, we show
that these Hopf algebras are pointed, and determine all Hopf
isomorphisms of $\mathfrak u_{r,s}(F_4)$. Section 6 is to prove that
$\mathfrak u_{r,s}(F_4)$ is of Drinfel'd double
under a certain condition. In Section 7, the left and right
integrals of $\mathfrak{b}$ are singled out based on the discussion
in Section 3, which are applied in Section 8 to determine the necessary and
sufficient conditions for $\mathfrak u_{r,s}(F_4)$
being ribbon.

\section{$U_{r,s}(F_4)$ and its convex PBW-type Lyndon basis}

\medskip
\noindent{\it 2.1.} According to the unified definition of two-parameter quantum groups in \cite{HP} using the Euler form, we can write down the defining relations of
two-parameter quantum group of type $F_4$.

Let $\mathbb{K}=\mathbb{Q}(r,s)$ be a subfield of  $\mathbb{C}$, where $r,s\in\mathbb{C}^*$ with assumptions  $r^2\neq
s^2, r^3\neq s^3$. Let $\Phi$ be  a root system of type $F_4$ with  $\Pi$
a base of  simple roots,  which is a finite subset of euclidian space
$E=\mathbb{R}^4$. Let $\epsilon_1, \epsilon_2,\epsilon_3,
\epsilon_4$ denote an orthonormal basis of $E$, then
$\Pi=\{\alpha_1=\epsilon_2-\epsilon_3,
\alpha_2=\epsilon_3-\epsilon_4, \alpha_3=\epsilon_4,
\alpha_4=\frac{1}{2}(\epsilon_1-\epsilon_2-\epsilon_3-\epsilon_4)\}$,
$\Phi=\{\pm\epsilon_i, \pm\epsilon_i\pm\epsilon_j,
\frac{1}{2}(\pm\epsilon_1\pm\epsilon_2\pm\epsilon_3\pm\epsilon_4)\}.$
Let $r_1=r_2=r^2, r_3=r_4=r, s_1=s_2=s^2, s_3=s_4=s$. Write the
matrix of structural constants as follows:
$$
\left(
\begin{array}{cccc}
r^2s^{-2}&s^2&1&1\\
r^{-2}&r^2s^{-2}&s^2&1\\
1&r^{-2}&rs^{-1}&s\\
1&1&r^{-1}&rs^{-1}
\end{array}
\right)=(a_{ij})_{4\times4}.
$$

\begin{defi}\label{D21}
Let $U=U_{r,s}(F_4)$ be the associative algebra over
$\mathbb{K}=\mathbb{Q}(r,s)$ generated by $E_i, F_i,
\omega_i^{\pm1}, \omega_i'^{\pm1} (1\leq i, j\leq4)$
subject to  relations $(F1)$---$(F6)$:
\begin{gather*}
\text{All $\omega_i^{\pm 1}, \omega_j'^{\pm 1}$'s commute with one
another and  $\omega_i\omega_i^{-1} =1= \omega_j'\omega_j'^{-1}$}. \tag{\text{$F1$}}\\
\omega_iE_j\omega_i^{-1}=a_{ij}E_j, \quad \omega_iF_j\omega_i^{-1}=a_{ij}^{-1}F_j. \tag{\text{$F2$}}\\
\omega_i^\prime E_j\omega_i'^{-1}=a_{ji}^{-1}E_j, \quad \omega_i^\prime F_j\omega_i'^{-1}=a_{ji}F_j.  \tag{\text{$F3$}}\\
[\,E_i,
F_j\,]=\delta_{ij}\frac{\om_i-\om_i'}{r_i-s_i}. \tag{\text{$F4$}}
\end{gather*}
\noindent $(F5)$ \, The following $(r,\,s)$-Serre
relations hold:
\begin{gather*}
E_1^2E_2-(r^2+s^2)E_1E_2E_1+r^2s^2E_2E_1^2=0,\tag*{$(F5)_1$}\\
E_2^2E_1-(r^{-2}+s^{-2})E_2E_1E_2+r^{-2}s^{-2}E_1E_2^2=0,\tag*{$(F5)_2$}\\
E_2^2E_3-(r^2+s^2)E_2E_3E_2+r^2s^2E_3E_2^2=0,\tag*{$(F5)_3$}\\
\begin{split}
&E_3^3E_2-(r^{-2}+r^{-1}s^{-1}+s^{-2})E_3^2E_2E_3\\  & \quad\qquad+r^{-1}s^{-1}(r^{-2}+r^{-1}s^{-1}+s^{-2})E_3E_2E_3^2-r^{-3}s^{-3}E_2E_3^3=0,
\end{split}\tag*{$(F5)_4$}\\
E_3^2E_4-(r+s)E_3E_4E_3+rsE_4E_3^2=0,\tag*{$(F5)_5$}\\
E_4^2E_3-(r^{-1}+s^{-1})E_4E_3E_4+r^{-1}s^{-1}E_3E_4^2=0.\tag*{$(F5)_6$}
\end{gather*}
\noindent $(F6)$ \,  The following $(r,\,s)$-Serre
relations hold:
\begin{gather*}
F_2F_1^2-(r^2+s^2)F_1F_2F_1+r^2s^2F_1^2F_2=0,\tag*{$(F6)_1$}\\
F_1F_2^2-(r^{-2}+s^{-2})F_2F_1F_2+r^{-2}s^{-2}F_2^2F_1=0,\tag*{$(F6)_2$}\\
F_3F_2^2-(r^2+s^2)F_2F_3F_2+r^2s^2F_2^2F_3=0,\tag*{$(F6)_3$}\\
\begin{split}
&F_2F_3^3-(r^{-2}+r^{-1}s^{-1}+s^{-2})F_3F_2F_3^2\\
& \quad\qquad+r^{-1}s^{-1}(r^{-2}+r^{-1}s^{-1}+s^{-2})F_3^2F_2F_3-r^{-3}s^{-3}F_3^3F_2=0,
\end{split}\tag*{$(F6)_4$}\\
F_4F_3^2-(r+s)F_3F_4F_3+rsF_3^2F_4=0,\tag*{$(F6)_5$}\\
F_3F_4^2-(r^{-1}+s^{-1})F_4F_3F_4+r^{-1}s^{-1}F_4^2F_3=0.\tag*{$(F6)_6$}\\
\end{gather*}
\end{defi}

The algebra $U_{r,s}(F_4)$ is a Hopf algebra, where all
$\om_i^{\pm1}, {\om_i'}^{\pm1}$'s  are group-like elements, and the
remaining Hopf structure is given by
\begin{gather*}
\Delta(E_i)=E_i\otimes1+\omega_i\otimes E_i,\quad  \Delta(F_i)=1\otimes F_i+F_i\otimes \omega_i^\prime,\\
\varepsilon(\omega_i^{\pm1})=\varepsilon(\omega_i^\prime)^{\pm1}=1, \quad \varepsilon(E_i)=\varepsilon(F_i)=0,\\
S(\om_i^{\pm1})=\om_i^{\mp1}, \qquad
S({\om_i'}^{\pm1})={\om_i'}^{\mp1},\\
S(E_i)=-\om_i^{-1}E_i,\qquad S(F_i)=-F_i\,{\om_i'}^{-1}.
\end{gather*}
When $r = q=s^{-1}$, the Hopf algebra $U_{r,\,s}(F_4)$ modulo the
Hopf ideal generated by $\omega_i' - \omega_i^{-1} (1\leqslant i\leqslant 4)$, is just
the quantum group $U_q(F_4)$ of Drinfel'd-Jimbo type.

In any Hopf algebra $\mathcal H$, there exist the left-adjoint and
the right-adjoint action defined by its Hopf algebra structure as
follows
$$
\text{ad}_{ l}\,a\,(b)=\sum_{(a)}a_{(1)}\,b\,S(a_{(2)}), \qquad
\text{ad}_{ r}\,a\,(b)=\sum_{(a)}S(a_{(1)})\,b\,a_{(2)},
$$
where $\Delta(a)=\sum_{(a)}a_{(1)}\ot a_{(2)}\in \mathcal H\otimes
\mathcal H$, for any $a$, $b\in \mathcal H$.

From the viewpoint of adjoint actions, the $(r,s)$-Serre relations
$(F5)$, $(F6)$ take the simplest forms:
\begin{gather*}
\bigl(\text{ad}_l\,E_i\bigr)^{1-c_{ij}}\,(E_j)=0,
\qquad\text{\it for any } \ i\ne j,\\
\bigl(\text{ad}_r\,F_i\bigr)^{1-c_{ij}}\,(F_j)=0, \qquad\text{\it
for any } \ i\ne j,
\end{gather*}
where $C=(c_{ij})_{4\times 4}$ is the Cartan matrix of type $F_4$.

\smallskip

\noindent {\it 2.2.} $U$ has a triangular decomposition
$U\cong U^-\otimes U^0\otimes U^+$, where $U^0$ is the subalgebra
generated by $\omega_i^{\pm 1}, {\omega'_i}^{\pm 1}$, and $U^+$
(resp. $U^-$) is the subalgebra generated by $E_i$
 (resp. $F_i$). Let $\mathcal{B}$ (resp. $\mathcal{B}'$)
denote the Hopf subalgebra of $U$ generated by $E_j, \omega_j^{\pm
1}$ (resp. $F_j, {\omega'_j}^{\pm 1}$) with $1\leq j\leq 4$.

\begin{prop} \label{P22}
There exists a unique skew-dual pairing $\langle\,,\rangle:$ $
\mathcal{B}'\times \mathcal{B}\rightarrow \mathbb{K}$ of the Hopf
subalgebras $\mathcal{B}$ and $\mathcal{B}'$ such that
\begin{gather*}
\langle F_i,\,E_j \rangle = \delta_{ij} \frac{1}{s_i - r_i},\\
\langle \omega_i',\,\omega_j \rangle =
a_{ji},\\
\langle {\omega'}^{\pm 1}_i, \omega^{-1}_j\rangle=\langle
{\omega'}^{\pm 1}_i, \omega_j\rangle^{-1}=
 \langle \omega_i', \omega_j\rangle^{\mp 1},
\end{gather*} for $1 \leq i,\,j \leq 4$
and all other pairs of generators are $0$. Moreover, we have
$\langle S(a),\,S(b) \rangle = \langle a,\,b \rangle$ for $a\in
\mathcal{B}', b\in \mathcal{B}$.
\end{prop}

\noindent {\it 2.3.}
Recall that a reduced expression of the longest element of Weyl
group $W$ for type $F_4$  (see \cite{CX}) taken as
$$
w_0=s_1s_2s_3s_2s_1s_4s_3s_2s_3s_4s_1s_2s_3s_2s_1s_2s_3s_2s_4s_3s_2s_4s_3s_4$$
yields a convex
ordering of positive roots.
Write the positive root system $\Phi^+ = \{\b_1,\b_2,\cdots, \b_{24}\}$ with the positive root $\b_i$  given in Table I.
We will work with this fixed choice of a convex ordering on  $\Phi^+$, which our  constructions will implicitly depend on.
%$\Phi^+ = \{\alpha_1,
%\alpha_1+\alpha_2, \alpha_1+\alpha_2+\a_3, \a_1+\a_2+2\a_3, \a_1+2\a_2+2\a_3,  \a_1+\a_2+\a_3+\a_4,  \a_1+\a_2+2\a_3+\a_4, 2 \a_1+3\a_2+4\a_3+2\a_4,
% \a_1+2\a_2+2\a_3+\a_4,  \a_1+2\a_2+3\a_3+\a_4,  \a_1+\a_2+2\a_3+2\a_4,  \a_1+2\a_2+2\a_3+2\a_4,  \a_1+2\a_2+3\a_3+2\a_4,  \a_1+2\a_2+4\a_3+2\a_4,  \a_1+3\a_2+4\a_3+2\a_4,  \a_2,  \a_2+\a_3,  \a_2+2\a_3,  \a_2+\a_3+\a_4,  \a_2+2\a_3+\a_4,  \a_2+2\a_3+2\a_4,  \a_3,  \a_3+\a_4, \a_4 \}$ and 24 convex ordering positive roots $\b_1,\b_2,\cdots, \b_{24}$.

Note that the above ordering also corresponds to the standard Lyndon
tree of type $F_4$ (see \cite{LR}):

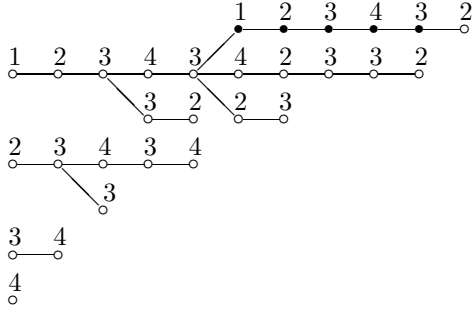
\unitlength 0.6mm % = 2.85pt
\linethickness{0.1pt}
\ifx\plotpoint\undefined\newsavebox{\plotpoint}\fi % GNUPLOT compatibility

\begin{center}
\begin{picture}(90,90)(0,-70)
\put(0,0){\circle{2}} \put(10,0){\circle{2}}\put(20,0){\circle{2}} \put(30,0){\circle{2}}\put(40,0){\circle{2}}\put(50,0){\circle{2}}
\put(60,0){\circle{2}}\put(70,0){\circle{2}}\put(80,0){\circle{2}}\put(90,0){\circle{2}}
\put(50,10){\circle*{2}}\put(60,10){\circle*{2}}\put(70,10){\circle*{2}}\put(80,10){\circle*{2}}\put(90,10){\circle*{2}}\put(100,10){\circle{2}}
\put(30,-10){\circle{2}}\put(40,-10){\circle{2}}
\put(50,-10){\circle{2}}\put(60,-10){\circle{2}}
\put(0,-20){\circle{2}}\put(10,-20){\circle{2}}\put(20,-20){\circle{2}}\put(30,-20){\circle{2}}\put(40,-20){\circle{2}}
\put(20, -30){\circle{2}}
\put(0, -40){\circle{2}}\put(10, -40){\circle{2}}
\put(0, -50){\circle{2}}
\put(1,0){\line(1,0){8}} \put(11,0){\line(1,0){8}} \put(21,0){\line(1,0){8}} \put(31,0){\line(1,0){8}} \put(41,0){\line(1,0){8}} \put(51,0){\line(1,0){8}}
\put(61,0){\line(1,0){8}} \put(71,0){\line(1,0){8}} \put(81,0){\line(1,0){8}}
\put(41,1){\line(1,1){8}}\put(51,10){\line(1,0){8}}\put(61,10){\line(1,0){8}}
\put(71,10){\line(1,0){8}}
\put(81,10){\line(1,0){8}}
\put(91,10){\line(1,0){8}}
\put(21,-1){\line(1,-1){8}}\put(31,-10){\line(1,0){8}}
\put(41,-1){\line(1,-1){8}}\put(51,-10){\line(1,0){8}}
\put(1,-20){\line(1,0){8}} \put(11,-20){\line(1,0){8}}
\put(21,-20){\line(1,0){8}}
\put(31,-20){\line(1,0){8}}
\put(11,-21){\line(1,-1){8}}
\put(1,-40){\line(1,0){8}}
\put(-1,2){1}\put(9,2){2}\put(19,2){3}\put(29,2){4}\put(39,2){3}\put(49,2){4}\put(59,2){2}\put(69,2){3}\put(79,2){3}\put(89,2){2}
\put(49,12){1}\put(59,12){2}\put(69,12){3}\put(79,12){4}\put(89,12){3}\put(99,12){2}
\put(29,-8){3}\put(39,-8){2}\put(49,-8){2}\put(59,-8){3}
\put(-1,-18){2}\put(9,-18){3}\put(19,-18){4}\put(29,-18){3}\put(39,-18){4}
\put(20,-28){3}\put(-1,-38){3}\put(9,-38){4}
\put(-1,-48){4}
\end{picture}
\end{center}

We can make an inductive definition of the quantum root vectors  $
E_{\b_i}$ in $U^+$. Write $E_{\b_i}=E_i\ (1\leqslant i\leqslant 4)$
if $\b_i $ is a simple root. Let $\gamma\in \Phi^+$ (not a simple
root), define a minimal pair for $\gamma$  to be a pair $(\a,\b)$ of
positive roots such that $\gamma=\a+\b$, $\a<\gamma<\b$ and there is
no pair of positive roots $(\a', \b')$ such that $\gamma=\a'+\b'$
and

\begin{center}
Table I:  Positive roots with the convex ordering and quantum root vectors
\end{center}

\bigskip

\begin{tiny}
\newcommand{\tabincell}[2]{\begin{tabular}{@{}#1@{}}#2\end{tabular}}

\begin{tabular}{|c|c|c|c|c|c|}

  \hline
   order  &\tabincell{c} {positive quantum\\ root vectors}  &  positive roots  &  \tabincell{c} {subscripts of\\
   std. Lyndon words $u_i$} &$h{(\b_i)}$& \tabincell{c} {${({M})}$=minimal pair\\ ${({L})}$=Lyndon pair} \\
  \hline
  1 & $E_1$ & $\b_1{=}\a_1$  & $u_1{=}1$ &1 &$ simple $\\
  \hline
  2 & $E_{12}$ & $\b_2{=}\a_1{+}\a_2$ &$u_2{=}12{=}1\; 2$ & 2&
$\b_2{=}\b_1{+}\b_{16}\ (M)$ \\
  \hline
 3 & $E_{123}$ & $\b_3{=}\a_1{+}\a_2{+}\a_3$ &\tabincell{c}{$u_3{=}123{=}12\; 3$\\${=}1\; 23$} &3&
\tabincell{c}{$ \b_3{=}\b_2{+}\b_{22}\ (M)$\\ ${=}\b_1{+}\b_{17}\ (L)$} \\
  \hline
  4 & $E_{1233}$ & $\b_4{=}\a_1{+}\a_2{+}2\a_3$ &\tabincell{c}{$u_4{=}1233{=}123\; 3$\\${=}1\; 233$} &4&
\tabincell{c}{$\b_4{=}\b_3{+}\b_{22}\ (M)$\\
${=}\b_1{+}\b_{18}\ (L)$} \\
  \hline
  5 & $E_{12332}$ & $\b_5{=}\a_1{+}2\a_2{+}2\a_3$ &$u_5{=}12332{=}1233\; 2$ &5&
     \tabincell{ccc} {$\b_5{=}\b_4{+}\b_{16}\ (M)$\\ ${=}\b_3{+}\b_{17}$ \\ ${=}\b_2{+}\b_{18}$} \\
  \hline
 6 & $E_{1234}$ & \tabincell{c}{$\b_6{=}\a_1{+}\a_2$\\${+}\a_3{+}\a_4$} &\tabincell{cc}{$u_6{=}1234{=}123\; 4$\\${=}1\; 234$\\${=}12\; 34$} &4&
     \tabincell{ccc} {$\b_6{=}\b_3{+}\b_{24}\ (M)$ \\${=}\b_2{+}\b_{23}\ (L)$ \\ ${=}\b_1{+}\b_{19}\ (L)$} \\
  \hline
 7 & $E_{12343}$ &\tabincell{c} {$\b_7{=}\a_1{+}\a_2$\\${+}2\a_3{+}\a_4$} &\tabincell{cc}{$u_7{=}12343{=}1234\; 3$\\${=}1\; 2343$} &5&
     \tabincell{cccc} {$\b_7{=}\b_6{+}\b_{22}\ (M)$\\ ${=}\b_4{+}\b_{24}$ \\ ${=}\b_3{+}\b_{23}$\\ ${=}\b_1{+}\b_{20}\ (L)$} \\
  \hline
 8 & $E_{12343123432}$ &\tabincell{c}{$\b_8{=}2\a_1{+}3\a_2$\\${+}4\a_3{+}2\a_4$} & \tabincell{ccc}{$ u_8{=}12343123432$\\${=}12343\; 123432$} &11&
     \tabincell{ccccccc} {$\b_8{=}\b_7{+}\b_{9}\ (M)$\\${=}\b_6{+}\b_{10}$ \\ ${=}\b_5{+}\b_{11}$\\ ${=}\b_4{+}\b_{12}$\\${=}\b_3{+}\b_{13}$\\ ${=}\b_2{+}\b_{14}$\\ ${=}\b_1{+}\b_{15}$} \\
  \hline
 9 & $E_{123432}$ &\tabincell{c}{ $\b_9{=}\a_1{+}2\a_2$\\${+}2\a_3{+}\a_4$} &$u_9{=}123432{=}12343\; 2$ &6&
     \tabincell{ccccc} {$\b_9{=}\b_7{+}\b_{16}\ (M)$ \\${=}\b_6{+}\b_{17}$ \\ ${=}\b_5{+}\b_{24}$\\ ${=}\b_3{+}\b_{19}$\\${=}\b_2{+}\b_{20}$} \\
  \hline
 10 & $E_{1234323}$ &\tabincell{c}{$\b_{10}{=}\a_1{+}2\a_2$\\${+}3\a_3{+}\a_4$} &\tabincell{cc}{$u_{10}{=}1234323$\\
 ${=}123432\; 3$} &7&
     \tabincell{cccccc} {$\b_{10}{=}\b_9{+}\b_{22}\ (M)$ \\$\quad {=}\b_7{+}\b_{17}\ (L)$ \\ ${=}\b_6{+}\b_{18}$\\ ${=}\b_5{+}\b_{23}$\\${=}\b_4{+}\b_{19}$\\ ${=}\b_3{+}\b_{20}$} \\
  \hline
11 & $E_{123434}$ & \tabincell{c}{$\b_{11}{=}\a_1{+}\a_2$\\${+}2\a_3{+}2\a_4$} & \tabincell{cccccc} {$u_{11}{=}123434{=}12343\; 4$\\${=}1234\; 34$\\${=}1\; 23434$} &6&
  \tabincell{ccc} {$\b_{11}{=}\b_7{+}\b_{24}\ (M)$\\ ${=}\b_6{+}\b_{23}\ (L)$\\ ${=}\b_1{+}\b_{21}\ (L)$}  \\
  \hline
  12 & $E_{1234342}$ & \tabincell{c}{$\b_{12}{=}\a_1{+}2\a_2$\\${+}2\a_3{+}2\a_4$} &\tabincell{ccc}{$u_{12}{=}1234342$\\${=}123434\; 2$} &7&
  \tabincell{cccccc} {$\b_{12}{=}\b_{11}{+}\b_{16}\ (M)$\\ ${=}\b_9{+}\b_{24}\ (L)$\\ ${=}\b_6{+}\b_{19}$}  \\
  \hline
13 & $E_{12343423}$ & \tabincell{c}{$\b_{13}{=}\a_1{+}2\a_2$\\${+}3\a_3{+}2\a_4$} &\tabincell{ccc}{$u_{13}{=}12343423$\\${=}1234342\; 3$} &8&
     \tabincell{cccccc} {$\b_{13}{=}\b_{12}{+}\b_{22}\ (M)$\\$\quad{=}\b_{11}{+}\b_{17}\ (L)$\\${=}\b_{10}{+}\b_{24}$ \\ ${=}\b_9{+}\b_{23}$\\ ${=}\b_7{+}\b_{19}$\\${=}\b_6{+}\b_{20}$} \\
  \hline
14 & $E_{123434233}$ & \tabincell{c}{$\b_{14}{=}\a_1{+}2\a_2$\\${+}4\a_3{+}2\a_4$} & \tabincell{cccccc} {$u_{14}{=}123434233$\\${=}12343423\; 3$\\${=}123434\; 233$} &9&
     \tabincell{cccccc} {$\b_{14}{=}\b_{13}{+}\b_{22}\ (M)$\\  $\quad\quad {=}\b_{11}{+}\b_{18}\ (L)$\\ ${=}\b_{10}{+}\b_{23}$\\${=}\b_7{+}\b_{20}$\\${=}\b_4{+}\b_{21}$} \\
  \hline
15 & $E_{1234342332}$ & \tabincell{c}{$\b_{15}{=}\a_1{+}3\a_2$\\${+}4\a_3{+}2\a_4$} & \tabincell{cccccc}{$u_{15}{=}1234342332$\\
${=}123434233\; 2$} &10&
     \tabincell{cccccc} {$\b_{15}{=}\b_{14}{+}\b_{16}\ (M)$\\${=}\b_{13}{+}\b_{17}$ \\ ${=}\b_{12}{+}\b_{18}$\\${=}\b_{10}{+}\b_{19}$\\${=}\b_9{+}\b_{20}$\\${=}\b_5{+}\b_{21}$} \\
  \hline
16 & $E_2$ & $\b_{16}{=}\a_2$  & $u_{16}{=}2$ &1 &$ simple $\\
  \hline
17 & $E_{23}$ & $\b_{17}{=}\a_2{+}\a_3$ &$u_{17}{=}23{=}2\; 3$ & 2&
$\b_{17}{=}\b_{16}{+}\b_{22}\ (M)$ \\
  \hline
 18 & $E_{233}$ & $\b_{18}{=}\a_2{+}2\a_3$ &$u_{18}{=}233{=}23\; 3$ &3&
$\b_{18}{=}\b_{17}{+}\b_{22}\ (M)$ \\
  \hline
 19 & $E_{234}$ & $\b_{19}{=}\a_2{+}\a_3{+}\a_4$ &     \tabincell{cccccc} {$u_{19}{=}234{=}23\; 4$\\${=}2\; 34$} &3&
\tabincell{cc} {$\b_{19}{=}\b_{17}{+}\b_{24}\ (M)$ \\${=}\b_{16}{+}\b_{23}\ (L)$} \\
  \hline
 20 & $E_{2343}$ & $\b_{20}{=}\a_2{+}2\a_3{+}\a_4$ &$u_{20}{=}2343{=}234\; 3$ &4&
\tabincell{ccc} {$\b_{20}{=}\b_{19}{+}\b_{22}\ (M)$\\${=}\b_{18}{+}\b_{24}$\\ ${=}\b_{17}{+}\b_{23}$}\\
  \hline
  21 & $E_{23434}$ & $\b_{21}{=}\a_2{+}2\a_3{+}2\a_4$ &     \tabincell{cccccc} {$u_{21}{=}23434{=}2343\; 4$\\{=}$234\; 34$} &5&
\tabincell{cc} {$\b_{21}{=}\b_{20}{+}\b_{24}\ (M)$\\ ${=}\b_{19}{+}\b_{23}\ (L)$}\\
  \hline
 22 & $E_3$ & $\b_{22}{=}\a_3$  & $u_{22}{=}3$ &1 &$ simple $\\
  \hline
  23 & $E_{34}$ & $\b_{23}{=}\a_3{+}\a_4$ &$u_{23}{=}34{=}3\; 4$ &2&
$\b_{23}{=}\b_{22}{+}\b_{24}\ (M)$\\
  \hline
 24 & $E_4$ & $\b_{24}{=}\a_4$  & $u_{24}{=}4$ &1 &$ simple $\\
  \hline
 \end{tabular}
\end{tiny}
\newpage\noindent
$\a<\a'<\gamma<\b'$. As for how to get the quantum root vector
$E_{\gamma}$ in $U^+$, we have to add $(r,s)$-bracketings on each
standard Lyndon word obeying the defining rule as
$E_\gamma:=[E_\alpha,
E_\beta]_{\langle\omega_\beta',\omega_\alpha\rangle}=E_\alpha
E_\beta-\langle \omega_\beta',\omega_\alpha\rangle E_\beta E_\alpha$
for $\alpha, \gamma,\beta\in\Phi^+$ with minimal
pair $(\alpha,\beta)$ for $\gamma$. Briefly, we denote these positive quantum root
vectors  with the above ordering as in Table I.
%$E_1, E_{12}, E_{123}, E_{1233}, E_{12332},
%E_{1234}, E_{12343}, E_{12343123432}, E_{123432}, E_{1234323}, E_{123434},E_{1234342}, $  $E_{12343423},
%E_{123434233}, E_{1234342332}, E_2, E_{23}, E_{233}, E_{234}, E_{2343}, E_{23434}
%E_{3}, E_{34}, E_{4}.$
Note that the first decomposition for each positive root in Table I
is a minimal pair.  It will be checked in Lemma \ref{indep of decom} below that the $(r,s)$-brackets for each standard Lyndon words is independent of the choice of the Lyndon decomposition. In particular, we have
$E_{\b_k}=[E_{\b_i},E_{\b_j}]_{\langle\omega_{\beta_j}',\omega_{\beta_i}\rangle}$
if the subscripts Lyndon word associated to $\beta_k$ decomposes into
the co-standard factorization $u_k=u_iu_j$.

For examples, $E_{\b_2}=E_{12}=
E_1E_2-\langle \omega_{\a_2}',\omega_{\a_1}\rangle E_2E_1=
E_1E_2-a_{12}E_2E_1=E_1E_2-s^2E_2E_1$, where root $\b_2$ has the only decomposition $\b_2=\a_1+\a_2$; while  $E_{\b_{20}}=E_{2343}=
E_{234}E_3-\langle \omega_{\a_3}',\omega_{\a_2+\a_3+\a_4}\rangle E_3E_{234}=
E_{234}E_3-a_{23}a_{33}a_{43}E_3E_{234}=E_{234}E_3-sE_3E_{234}$, where the root decomposition $\a_2+2\a_3+\a_4=(\a_2+\a_3+\a_4)+\a_3$, corresponding to the co-standard factorization of
 Lyndon words, i.e., $234\, 3$.
We find that $\b_{20}$ has the other two decompositions
$\b_{20}=(\a_2+\a_3)+(\a_3+\a_4)=(\a_2+2\a_3)+\a_4$, which result in more complicated relations among quantum root vectors $E_{2343}, E_{23}, E_{34}, E_{233}, E_{4}$.

According to \cite{K, Ro, L}, we have

\begin{theorem}\label{BaseE}
$\left.\Bigl\{\, E_{\b_{24}}^{n_{24}}E_{\b_{23}}^{n_{23}}\cdots
E_{\b_{1}}^{n_{1}}
\,\right| \b_1<\b_2<\cdots<\b_{23}<\b_{24}\in \Phi^+,
n_{i}\in  \textit{N}\Bigr\}$ forms a convex
PBW-type Lyndon basis of the algebra $U^+$.\hfill\qed
\end{theorem}

\begin{defi}\label{d24}
Let $\tau$ be the $\mathbb{Q}$-algebra anti-automorphism of
$U_{r,s}(F_4)$ such that $\tau(r)=s$, $\tau(s)=r$,
$\tau(\langle \om_i',\om_j\rangle^{\pm1})=\langle
\om_j',\om_i\rangle^{\mp1}$, and
\begin{gather*}
\tau(E_i)=F_i, \quad \tau(F_i)=E_i, \quad \tau(\om_i)=\om_i',\quad
\tau(\om_i')=\om_i.
\end{gather*}
Then $\mathcal B'=\tau(\mathcal B)$ with those induced defining
 relations from $\mathcal B$, and those cross relations in
$(\textrm{F2})$---$(\textrm{F4})$ are antisymmetric with respect to
$\tau$.\hfill\qed
\end{defi}

Using $\tau$ to $U^+$, we get those negative quantum root
vectors in $U^-$: $F_1, F_{12}, F_{123}$, $F_{1233}, F_{12332},
F_{1234}, F_{12343}, F_{12343123432}, F_{123432}, F_{1234323}, F_{123434}, F_{1234342}, F_{12343423}$,  $
F_{123434233}, F_{1234342332}, F_2, F_{23}, F_{233}, F_{234}, F_{2343}, F_{23434},
F_{3}, F_{34}, F_{4}$. We then have the following

\begin{coro}\label{BaseF}
$\left.\Bigl\{\, F_{\b_{1}}^{n_{1}}F_{\b_{2}}^{n_{2}}\cdots
F_{\b_{24}}^{n_{24}}
\,\right|\, \b_1<\b_2<{\cdots}<\b_{24}\in \Phi^+,
n_{i}\in  \textit{N}\Bigr\}$ forms
a convex PBW-type basis of the algebra $U^-$.\hfill\qed
\end{coro}

%\smallskip

%%
\section{Commutation relations and homogeneous central elements}
\medskip
\noindent {\it 3.1.  Some  formulas.}
Let $Q^+=\sum\limits_{i=1}^4\mathbb{Z}^+\alpha_i$, for each
$\zeta=\sum\limits_{i=1}^4\zeta_i\alpha_i\in Q^+$, denote
$$U_\zeta^+=\{\,e\in U^+\mid\omega_{\eta} e\omega_{\eta}^{-1}=\langle\omega_{\zeta}^\prime,\omega_{\eta}\rangle
e,\quad \omega_{\eta}^\prime e
\omega_{\eta}'^{-1}=\langle\omega_{\eta}^\prime,\omega_{\zeta}\rangle^{-1}e,\quad
\forall\; \eta\in\sum\limits_{i=1}^4\mathbb{Z}\alpha_i\,\}.$$ For
each $x\in U_\zeta^+, y\in U_{\eta}^+$, we defined
$[x,y]_{p_{xy}}=xy-p_{xy}yx$, where
$p_{xy}=\langle\omega_\eta^\prime,\omega_\zeta\rangle$ in subsection 2.3. We shall
write $[x,y]_{p_{xy}}$ simply as $[x,y]$ in this paper.

We extend the above operation to $U^+$ by natural linear extension.
According to \cite{K}, we have the following generalized Jacobi
identity and Leibniz rules, which will be applied in the
calculations of commutation relations.
\begin{gather*}
\begin{split}
[[x,y]_{p_{xy}},z]_{p_{xz}p_{yz}}&=[x,[y,z]_{p_{yz}}]_{p_{xy}p_{xz}}+
p_{zy}^{-1}[[x,z]_{p_{xz}},y]_{p_{xy}p_{zy}}\\& \quad +(p_{yz}-p_{zy}^{-1})[x,z]_{p_{xz}}y, \\
\end{split}\tag*{$(J)$}\label{J}
\end{gather*}
\begin{gather*}
\begin{split}
[x,yz]_{p_{xy}p_{xz}}&=[x,y]_{p_{xy}}z+p_{xy}y[x,z]_{p_{xz}},\\
[xy,z]_{p_{xz}p_{yz}}&=p_{yz}[x,z]_{p_{xz}}y+x[y,z]_{p_{yz}}.
\end{split}\tag*{$(L)$}\label{L}
\end{gather*}

\begin{lemm}\label{indep of decom}
The $(r,s)$-brackets for each standard Lyndon words is independent of the choice of the Lyndon decomposition.
\end{lemm}
\begin{proof}[{ \bf Proof.}]
It is enough to check this for $E_{\beta_3}=E_{123}$, $E_{\beta_4}=E_{1233}$, $E_{\beta_6}=E_{1234}$, $E_{\beta_7}=E_{12343}$, $E_{\beta_{11}}=E_{123434}$,  $E_{\beta_{14}}=E_{123434233}$, $E_{\beta_{19}}=E_{234}$, $E_{\beta_{21}}=E_{23434}$, since the others are either simple or have unique Lyndon decomposition.

For $E_{\beta_3}$, we have to show that $[E_1,E_{23}]=[E_{12},E_3]$. We have $[E_{12},E_3]=[[E_1,E_2],E_3]=[E_1,[E_2,E_3]]=[E_1,E_{23}]$, this is clear by $[E_1,E_3]=0$ and \ref{J}.

For $E_{\beta_4}$, we have to show $[E_{123},E_3]=[E_1,E_{233}]$. By the statement for $E_{\beta_3}$, we have
$$[E_{123},E_3]=[[E_1,E_{23}],E_3]=[E_1,[E_{23},E_3]]=[E_1,E_{233}]$$
by $[E_1,E_3]=0$ and \ref{J}.

For $E_{\beta_{19}}$,  we have $[E_{23},E_4]=[E_2,E_{34}]$ similarly by $[E_2,E_4]=0$ and \ref{J}.

For $E_{\beta_6}$, we have to show $[E_{123},E_4]=[E_{12},E_{34}]=[E_1,E_{234}]$. By the construction of $E_{\beta_3}$, we have $[E_{123},E_4]=[[E_{12},E_3],E_4]=[E_{12},E_{34}]$, since $[E_{12},E_4]=0$ and \ref{J}. By the construction of  $E_{\beta_{19}}$, we have $[E_{12},E_{34}]=[E_1,[E_2,E_{34}]]=[E_1,E_{234}]$, since $[E_1,E_{34}]=0$ and \ref{J}.

$E_{\beta_7}$,  we have to show that $[E_{1234},E_3]=[E_1,E_{2343}]$. By the construction of $E_{\beta_6}$, we have $[E_{1234},E_3]=[[E_1,E_{234}],E_3]=[E_1,[E_{234},E_3]]$, since $[E_1,E_3]=0$ and \ref{J}.

For $E_{\beta_{21}}$, we have to show $[E_{2343},E_4]=[E_{234},E_{34}]$. By \ref{J}, it is enough to show that $[E_{234},E_4]=0$. In fact, $[E_{34},E_4]=0$ by Serre relation (F5), hence $[E_{234},E_4]=[E_2,[E_{34},E_4]]=0$ by \ref{J}.

For $E_{\beta_{11}}$, we have to show that $[E_{12343},E_4]=[E_{1234},E_{34}]=[E_1,E_{23434}]$. To show the first equality, it is enough to show that $[E_{1234},E_4]=0$ by \ref{J} and the construction of $E_{\beta_7}$. By \ref{J}, we have $[E_{1234},E_4]=[[E_1,E_{234}],E_4]=[E_1,[E_{234},E_4]]=0$, since $[E_{234},E_4]=0$, which is already proved. By the construction of $E_{\beta_7}$ and $E_{\beta_{21}}$, we have $[E_{12343},E_4]=[[E_1,E_{2343}],E_4]=[E_1,[E_{2343},E_4]]=[E_1,E_{23434}]$, since $[E_1,E_4]=0$.

For $E_{\beta_{14}}$, we have to show that $[E_{12343423},E_3]=[E_{123434},E_{233}]$. By \ref{J} and $[E_1,E_3]=0$, it is enough to show that $[E_{123434},E_3]=0$. By the statement for $E_{\beta_{11}}$ and \ref{J}, we have $[E_{123434},E_3]=[[E_1,E_{23434}],E_3]=[E_1,[E_{23434},E_3]]$. To prove our claim, it suffices to show that $[E_{23434},E_3]=0$. It follows from \ref{J} that $[E_{23},E_{34}]=rE_{234}E_3-s^2E_3E_{234}$. On the one hand, we have
\begin{equation*}
\begin{split}
[E_{233},E_{34}]&=s[[E_{23},E_{34}],E_3]+(r-s)[E_{23},E_{34}]E_3\\
&=r[E_{23},E_{34}]E_3-rsE_3[E_{23},E_{34}]\\
&=r^2E_{234}E_3^2-rs^2E_3E_{234}E_3-r^2sE_3E_{234}E_3+rs^3E_3^2E_{234}\\&
=r^2E_{2343}E_3-rs^2E_3E_{2343},
\end{split}
\end{equation*}
by $[E_3,E_{34}]=0$.

On the other hand, we have
\begin{equation*}
\begin{split}
[E_{233},E_{34}]
&=-r[[E_{233},E_4],E_3]+(r-s)[E_{233},E_4]E_3\\
&=-r(r+s)[E_{2343},E_3]+(r^2-s^2)E_{2343}E_3\\
&=-s(r+s)E_{2343}E_3+r^2(r+s)E_3E_{2343}.
\end{split}
\end{equation*}

%By comparison, we have:
%$$r^2E_{2343}E_3-rs^2E_3E_{2343}=-s(r+s)E_{2343}E_3+r^2(r+s)E_3E_{2343},$$
%and then $(r^2+rs+s^2)[E_{2343},E_3]=0,$
So we get  $[E_{2343},E_3]=0$, since $r^2+rs+s^2\neq0$.

It follows from \ref{J} and $[E_3,E_{34}]=0$ that
\begin{equation*}
\begin{split}
[E_{2343},E_{34}]&=s[E_{23434},E_3]+(r-s)E_{23434}E_3\\
&=rE_{23434}E_3-sE_3E_{23434},
\end{split}
\end{equation*}
and
\begin{equation*}
\begin{split}
[E_{2343},E_{34}]
&=-r[E_{23434},E_3]+(r-s)E_{23434}E_3\\
&=-sE_{23434}E_3+rE_3E_{23434}.
\end{split}
\end{equation*}
%By comparison, we get:
%$rE_{23434}E_3-sE_3E_{23434}=-sE_{23434}E_3+rE_3E_{23434},$ and then
%$$(r+s)[E_{23434},E_3]=0.$$
Therefore $[E_{23434},E_3]=0$, since $r+s\neq0.$

This completes the proof.
\end{proof}

\smallskip

\noindent {\it 3.2. Commutation relations in $U^{+}$.}
We make
more efforts to explicitly calculate some commutation
relations, which are useful for determining central elements in
subsection {\it 3.3} and  the integrals in Section 7.

\begin{lemm} \label{l1} The following relations hold in $U^{+}:$

\smallskip

$(1)\quad [E_{1},E_{12}]=[E_{12},E_2]=[E_2,E_{23}]=[E_{233},E_{3}]=[E_{3},E_{34}]=[E_{34},E_{4}]=0;
$

$(2)\quad [E_{1233},E_{23}]=[E_{123},E_2]=[E_1,E_{123}]=[E_1,E_{1233}]=0;
$

$(3)\quad [E_{123},E_{23}]=E_2E_{1233}-s^2E_{1233}E_2;
$

$(4)\quad [E_{12},E_{123}]=[E_{23},E_{233}]=[E_{123},E_{1233}]=[E_{1233},E_{12332}]=0;$

$(5)\quad [E_{2},E_{233}]=r(r-s)E_{23}^2;$

$(6)\quad [E_{12},E_{23}]=(r^2-s^2)E_{123}E_2;$

$(7)\quad [E_{12},E_{233}]=r^2s^2E_{12332}+r(r-s)(E_{123}E_{23}+s^2E_{23}E_{123}).$
\end{lemm}
\begin{proof}[{\bf Proof.}]
(1)  follows directly from the $(r,s)$-Serre relation $(F5)$.

Note that $B_3$ is the subsystem of $F_4$, the commutator relations in $B_3$
still hold in $F_4$, therefore, we get  $(2)$---$(7)$ by
\cite{HW2}.
\end{proof}

Combining   \ref{J}  and  Lemma \ref{l1},
we get the  following lemmas:

\begin{lemm}\label{l2} The following relations hold in $U^{+}:$

\smallskip

      $(1) \quad [E_{1234},E_4]=[E_{1234},E_2]=[E_{1},E_{1234}]=[E_{12},E_{1234}]=[E_{1},E_{12343}]=0;$

      $(2)\quad [E_{23},E_{34}]=rE_{234}E_3-s^2E_3E_{234};$

      $(3)\quad [E_{233},E_4]=(r+s)E_{2343};$

      $(4)\quad [E_{1233},E_4]=(r+s)E_{12343};$

      $(5)\quad [E_{1234},E_{233}]=[E_{1233},E_{234}]=(r+s)E_{23}E_{12343}-s(r+s)E_{12343}E_{23};$

      $(6)\quad [E_{2},E_{2343}]=r^2E_{234}E_{23}-sE_{23}E_{234}=r(r-s)E_{234}E_{23};$

      $(7)\quad [E_{234},E_4]=[E_{2},E_{234}]=[E_{23},E_{234}]=0;$

      $(8)\quad [E_{12332},E_3]=0;$

      $(9)\quad [E_{123},E_{34}]=rE_{1234}E_3-s^2E_3E_{1234};$

      $(10)\quad [E_{1234},E_{23}]=E_2E_{12343}-s^2E_{12343}E_2;$

      $(11)\quad [E_{12332},E_{4}]=(r+s)E_{123432};$

     $(12)\quad [E_{12332},E_{34}]=r^2s^2(r+s)[E_{3},E_{123432}];$

      $(13)\quad [E_{123},E_{2343}]=r^2sE_{12343}E_{23}-s^2E_{23}E_{12343}+(r-s)(E_{1233}E_{234}+s^2E_{233}E_{1234});$

      $(14)\quad [E_{1234},E_{234}]=[E_{2},E_{123434}];$

      $(15)\quad [E_{1234},E_{2343}]=rE_{12343}E_{234}-s^2E_{234}E_{12343}.$

\end{lemm}

\begin{lemm} \label{l3} The following relations hold in $U^{+}:$

\smallskip

$(1)\quad [E_{233},E_{234}]=0;$

$(2)\quad [E_{23},E_{2343}]=(r-s)E_{233}E_{234};$

$(3)\quad [E_{234},E_{2343}]=0;$

$(4)\quad [E_{2343},E_3]=0;$

 $(5)\quad [E_{23434},E_3]=0;$

 $(6)\quad [E_{23434},E_4]=0;$

 $(7)\quad [E_{234},E_{23434}]=0;$

 $(8)\quad [E_{2343},E_{23434}]=0;$

 $(9)\quad [E_{233},E_{2343}]=0;$

 $(10)\quad [E_{233},E_{23434}]=r(r^2-s^2)E_{2343}^2;$

 $(11)\quad [E_{1233},E_{23434}]=r(r^2-s^2)(E_{12343}E_{2343}+s^2E_{2343}E_{12343})+r^2s^2E_{123434233}.$
\end{lemm}
\begin{proof}[{\bf Proof.}]
(1) \& (2): Observe that:
\begin{equation*}
\begin{split}
[E_{23},E_{2343}]&=-r[E_{233},E_{234}]+(r-s)E_{233}E_{234}\\
&=rE_{234}E_{233}-sE_{233}E_{234}.
\end{split}
\end{equation*}
Using  Lemma \ref{l2} (3) and \ref{J}, we have another expansion:
\begin{equation*}
\begin{split}
[E_{23},E_{2343}]
&=(r+s)^{-1}[E_{23},[E_{233},E_4]]\\
&=(r+s)^{-1}(-r^2[E_{234},E_{233}]+(r^2-s^2)E_{234}E_{233})\\
&=(r+s)^{-1}(r^2E_{233}E_{234}-s^2E_{234}E_{233}).
\end{split}
\end{equation*}

Combining the above two expansions, we have
 $[E_{233},E_{234}]=0$, since $r^2+rs+s^2\neq0$.

(3): Using  Lemma \ref{l2} (3),   we have
\begin{equation*}
[E_{234},E_{2343}]=(r+s)^{-1}[E_{234},[E_{233},E_4]],
\end{equation*}
which is $0$ by \ref{J}, (2) \& Lemma \ref{l2} (7).

(4) and (5) are already proved in the proof of Lemma \ref{indep of decom}.

(6) follows directly from \ref{l2} (7) and
$[E_{34}, E_4]=0$.

(7) follows directly from (3) and Lemma \ref{l2} (7).

(8): Noting that
$[E_{2343},E_{23434}]=[[E_{234},E_3],E_{23434}]$, and using
\ref{J}, (5) and (7), we obtain $[E_{2343},E_{23434}]=0$.

(9) follows directly from  (2) and Lemma \ref{l1} (1).

(10) follows from (9) and Lemma \ref{l2} (3).

(11) follows  from  \ref{J} and  (10).
\end{proof}

\begin{lemm}\label{l4} The following relations hold in $U^{+}:$

\smallskip
$(1)\quad [E_{123},E_{234}]=rE_{1234}E_{23}-s^2E_{23}E_{1234};$

 $(2)\quad [E_{123},E_{1234}]=0;$

 $(3)\quad [E_{1233},E_{1234}]=0;$

 $(4)\quad [E_{12332},E_{1234}]=0;$

 $(5)\quad [E_{1234},E_{12343}]=0;$

 $(6)\quad [E_{12343},E_{233}]=0;$

 $(7)\quad [E_{1234},E_{123432}]=0;$

 $(8)\quad [E_{1234},E_{1234323}]=rE_{123432}E_{12343}-rs^2E_{12343}E_{123432};$

 $(9)\quad [E_{12343123432},E_3]=[E_{12343},E_{1234323}]=0;$

 $(10)\quad [E_{1234},E_{12343123432}]=0;$

 $(11)\quad [E_{12343},E_{12343123432}]=0;$

  $(12)\quad [E_{12332},E_{12343}]=0;$

  $(13)\quad [E_{1233},E_{12343}]=0;$

  $(14)\quad [E_{1234323},E_{3}]=0;$

   $(15)\quad [E_{12},E_{2343}]=rs^2(r-s)E_{234}E_{123}+r(r-s)E_{1234}E_{23}+r^2s^2E_{123432};$

   $(16)\quad [E_{12343},E_{2343}]=(r+s)^{-1}(E_{233}E_{123434}-s^2E_{123434}E_{233});$

  $(17)\quad [E_{1233},E_{123434}]=r(r^2-s^2)E_{12343}^2;$

   $(18)\quad [E_{1233},E_{1234342}]=r(r^2-s^2)(r^{-2}E_{123432}E_{12343}+E_{12343}E_{123432})$

 \hskip4cm  $+\,r^2s^2E_{123434}E_{12332}-r^{-2}E_{12332}E_{123434};$

   $(19)\quad [E_{12332},E_{123434}]=-rs(r+s)[E_{123432},E_{12343}].$
\end{lemm}
\begin{proof}[{\bf Proof.}]
(1): Using \ref{J}, we have
\begin{equation*}
\begin{split}
[E_{123},E_{234}]&=[[E_{123},E_{23}],E_4]-r[E_{1234},E_{23}]+(r-s)E_{1234}E_{23}\\
&=[[E_2,E_{1233}],E_4]+rsE_{23}E_{1234}-sE_{1234}E_{23}\quad  (\mathrm{Lemma}\
\ref{l1}\ (2))\\
&=(r+s)[E_2,E_{12343}]+rsE_{23}E_{1234}-sE_{1234}E_{23}\quad  (\mathrm{Lemma}\
\ref{l2}\ (4))\\
&=rE_{1234}E_{23}-s^2E_{23}E_{1234} \quad (\mathrm{Lemma}\
\ref{l2}\ (10)).
\end{split}
\end{equation*}

(2): We have
\begin{equation*}
\begin{split}
[E_{123},E_{1234}]&=[[E_{123},E_1],E_{234}]-r^2[[E_{123},E_{234}],E_1]-(s^2{-}r^2)[E_{123},E_{234}]E_1\\
&=(1{-}r^2s^{-2})[E_{123}E_1,E_{234}]+s^{-2}E_1[E_{123},E_{234}]-s^2[E_{123},E_{234}]E_1\\
&=(1{-}r^2s^{-2})E_{123}E_{1234}-r^2[E_{123},E_{234}]E_1+s^{-2}E_1[E_{123},E_{234}]\\
&=(1{-}r^2s^{-2})E_{123}E_{1234}+s^{-2}[E_1,[E_{123},E_{234}]]\\
&=(1{-}r^2s^{-2})E_{123}E_{1234}+s^{-2}(r{+}s)[E_{12},E_{12343}]+rs^{-1}[E_1,E_{23}E_{1234}]\\
&\quad -s^{-1}[E_1,E_{1234}E_{23}]\\
&=(1{-}r^2s^{-2})E_{123}E_{1234}+s^{-2}(r{+}s)(-r[E_{123},E_{1234}]+(r{-}s)E_{123}E_{1234})\\
&\quad+rs^{-1}E_{123}E_{1234}-r^2s^{-1}E_{1234}E_{123}.
\end{split}
\end{equation*}
Therefore,
$[E_{123},E_{1234}]=0.$

(3): We have
\begin{equation*}
\begin{split}
&[E_{1233},E_{1234}]\\
&=(1-r^2s^{-2})[E_{1233}E_1,E_{234}]-r^2[[E_{1233},E_{234}],E_1]+(r^2-s^2)[E_{1233},E_{234}]E_1\\
&=(1-r^2s^{-2})E_{1233}E_{1234}+s^{-2}[E_1,[E_{1233},E_{234}]]\\
&=(1-r^2s^{-2})E_{1233}E_{1234}+s^{-2}(r+s)(E_{123}E_{12343}-r^2sE_{12343}E_{123})\\
&=(1-r^2s^{-2})E_{1233}E_{1234}+s^{-2}(r+s)[E_{123},E_{12343}]\\
&=(1-r^2s^{-2})E_{1233}E_{1234}-s^{-2}(r+s)(r[E_{1233},E_{1234}]+(r-s)E_{1233}E_{1234})\\
&=-rs^{-2}(r+s)[E_{1233},E_{1234}]
\end{split}
\end{equation*}
by (2) and Lemma \ref{l2} (1), (5).
%so  $(1+rs^{-1}+r^2s^{-2})[E_{1233},E_{1234}]=0$.
Therefore,
$[E_{1233},E_{1234}]=0$, since $1+rs^{-1}+r^2s^{-2}\neq0$.

(4) follows directly from (3) and Lemma \ref{l2} (1).

(5): Using  Lemma \ref{l2} (4), we have
$$[E_{1234},E_{12343}]=(r+s)^{-1}[E_{1234},[E_{1233},E_4]],$$
which is $0$, by  (3) and Lemma \ref{l2}  (1).

(6): Using $[E_{233},E_3]=0$, we have
\begin{equation*}
\begin{split}
&[E_{12343},E_{233}]\\
&=(1-r^2s^{-2})[E_{1234},E_3E_{233}]+r^{-2}[[E_{1234},E_{233}],E_3]+(s^{-2}-r^{-2})[E_{1234},E_{233}]E_3\\
&=(1-r^2s^{-2})E_{12343}E_{233}+s^{-2}[[E_{1234},E_{233}],E_3]\\
&=(1-r^2s^{-2})E_{12343}E_{233}-s^{-1}(r+s)E_{12343}E_{233}+rs^{-2}(r+s)E_{233}E_{12343}\\
&=-rs^{-2}(r+s)[E_{12343},E_{233}]
\end{split}
\end{equation*}
and  $(1+rs^{-1}+r^2s^{-2})[E_{12343},E_{233}]=0$.  So  $[E_{12343}, E_{233}]=0$.

(7) follows directly from (5) and Lemma \ref{l2} (1).

(8) follows directly from  \ref{J} and (7).

(9): Using (6) and (8), we have
\begin{equation*}
\begin{split}
&[E_{12343},E_{1234323}]\\
&=(1-rs^{-1})[E_{1234},E_3E_{1234323}]+r^{-2}s^{-1}[[E_{1234},E_{1234323}],E_3]\\
&\quad +(s^{-2}r^{-1}-r^{-2}s^{-1})[E_{1234},E_{1234323}]E_3\\
&=(1-rs^{-1})E_{12343}E_{1234323}+r^{-1}s^{-2}[[E_{1234},E_{1234323}],E_3]\\
&=(1-rs^{-1})E_{12343}E_{1234323}-rs^{-1}[E_{12343},E_{1234323}]+(rs^{-1}-1)E_{12343}E_{1234323}\\
&=-rs^{-1}[E_{12343},E_{1234323}],
\end{split}
\end{equation*}
 then $(1+rs^{-1})[E_{12343},E_{1234323}]=0$. Therefore, $[E_{12343}, E_{1234323}]=0$.

Using \ref{J}   and $[E_{12343},E_3]=0$,  we have:
$$[E_{12343123432},E_3]=[E_{12343},E_{1234323}].$$

(10)-(19): Using \ref{J}, (10) follows  from  (5)   and (7);
(11) follows directly from  (9) and  (10);
(12) follows directly from  (4)  and  Lemma \ref{l2}  (8);
(13) follows directly from  (3) and $[E_{1233},E_{3}]=0$;
(14) follows directly from  (6) and  $[E_{12343},E_{3}]=0$;
(15) follows directly from (1) and Lemma \ref{l2} (6);
(16) follows directly from (6) and  Lemma \ref{l2} (3);
(17) follows directly from  (13) and  Lemma \ref{l2} (4);
(18) follows directly from (17);  (19) follows directly from  (12) and  Lemma \ref{l2} (11).
\end{proof}

\begin{lemm}\label{l5} The following relations hold in $U^{+}:$
\smallskip

$(1)\quad [E_{123432},E_2]=0;$

$(2)\quad [E_{12343123432},E_2]=r^{-1}s^{-2}(r-s)E_{123432}^2;$

$(3)\quad [E_{12343123432},E_{123432}]=0;$

$(4)\quad [E_{1234323},E_2]=0;$

$(5)\quad [E_{123432},E_{1234323}]=0;$

$(6)\quad [E_{12343},E_{234}]=(r+s)^{-1}(E_{23}E_{123434}-s^2E_{123434}E_{23});$

$(7)\quad [E_{1234323},E_4]=rs(r+s)^{-1}E_{12343423};$

$(8)\quad [E_{123432},E_{34}]=s^2(r+s)^{-1}(r^2E_{3}E_{1234342}-E_{1234342}E_{3});$

$(9)\quad [E_{1234},E_{123434}]=0;$

$(10)\quad [E_{123432},E_{123434}]=0;$

$(11)\quad [E_{1234323},E_{123434}]=0;$

$(12)\quad [E_{12343},E_{123434}]=0;$

$(13)\quad [E_{1234342},E_4]=0;$

$(14)\quad [E_{12343},E_{1234342}]=(s^{-2}-r^{-2})E_{123432}E_{123434};$

$(15)\quad [E_{123434},E_{1234342}]=0;$

$(16)\quad [E_{1234323},E_{34}]=r^{2}sE_{3}E_{12343423}-sE_{12343423}E_{3};$

$(17)\quad [E_{123432},E_{234}]=(r+s)^{-1}(r^2E_{23}E_{1234342}-s^2E_{1234342}E_{23}).$
\end{lemm}
\begin{proof}[{\bf Proof.}]
(1): Using  \ref{J} and Lemma \ref{l2} (1), we have
\begin{equation*}
\begin{split}
[E_{123432},E_2]
%&=[[[E_1,E_{2343}],E_2],E_2]\\
&=[[[E_{1234},E_3],E_2],E_2]\\
&=[[E_{1234},[E_3,E_2]],E_2]\\
&=[E_{1234},[[E_3,E_2],E_2]]\\
&=0.
\end{split}
\end{equation*}

(2): Using  \ref{J} and  (1), we have
\begin{equation*}
\begin{split}
[E_{12343123432},E_2]&=r^{-2}[E_{123432},E_{123432}]+(s^{-2}-r^{-2})E_{123432}^2\\&=r^{-1}s^{-2}(r-s)E_{123432}^2.
\end{split}
\end{equation*}

(3): %Since
%$[E_2,E_{12343123432}]=-r^2s^2[E_{12343123432},E_2]=r(s-r)E_{123432}^2,$
Observe that
$$[E_{123432},E_{12343123432}]=E_{123432}E_{12343123432}-r^2E_{12343123432}E_{123432},$$
and
\begin{equation*}
\begin{split}
[E_{123432},E_{12343123432}]
&=[E_{12343},[E_2,E_{12343123432}]] \quad (\mathrm{(J)\  and\  Lemma}\  \ref{l4} (11))\\
&=[E_{12343}, -r^2s^2[E_{12343123432},E_2]]\\
&=r(s{-}r)[E_{12343},E_{123432}^2]  \quad (2)\\
&=r(s{-}r)E_{12343123432}E_{123432}+s^{-1}(s{-}r)E_{123432}E_{12343123432}
\quad (L).
\end{split}
\end{equation*}
Therefore,
%\begin{equation*}
%\begin{split}
%&E_{123432}E_{12343123432}-r^2E_{12343123432}E_{123432}
%\\&=r(s-r)E_{12343123432}E_{123432}+s^{-1}(s-r)E_{123432}E_{12343123432}
%\end{split}
%\end{equation*}
%and then
$[E_{12343123432},E_{123432}]=0$.

(4): Using (1) and $[E_2,E_{23}]=0$, we have
\begin{equation*}
\begin{split}
&[E_{1234323},E_2]
\\&=(1-r^2s^{-2})[E_{12343},E_{23}E_2]+r^{-2}[E_{123432},E_{23}]+(s^{-2}-r^{-2})E_{123432}E_{23}\\
&=(1-r^2s^{-2})E_{1234323}E_2+s^{-2}[E_{123432},E_{23}]\\
&=(1-r^2s^{-2})E_{1234323}E_2-r^2s^{-2}[E_{1234323},E_{2}]+(r^2s^{-2}-1)E_{1234323}E_2\\
&=-r^2s^{-2}[E_{1234323},E_{2}].
\end{split}
\end{equation*}
%so $(1+r^2s^{-2})[E_{1234323},E_{2}]=0$.
Therefore, $[E_{1234323},E_{2}]=0$, since  $1+r^2s^{-2}\neq0$.

(5) follows directly from (4), Lemma \ref{l4} (9) and
\ref{J}.

(6) \&  (7):  On the one hand,
\begin{equation*}
[E_{12343},E_{234}]=[E_{1234323},E_4]+r^{-1}E_{23}E_{123434}-sE_{123434}E_{23};
\end{equation*}
On the other hand,
\begin{equation*}
\begin{split}
 &[[E_{1233},E_4],E_{234}]\\
&=(1{-}rs^{-1})[E_{1233},E_4E_{234}]+r^{-1}[[E_{1233},E_{234}],E_4]+(s^{-1}{-}r^{-1})[E_{1233},E_{234}]E_4\\
&=(1{-}rs^{-1})[E_{1233},E_4]E_{234}+s^{-1}[[E_{1233},E_{234}],E_4]\\
&=(1{-}rs^{-1})(r{+}s)E_{12343}E_{234}-(r{+}s)[E_{1234323},E_4]+(r^2{-}s^2)E_{234}E_{12343}\\
 &\quad +(rs)^{-1}(r^2{-}s^2)E_{23}E_{123434}\\
&=(1{-}rs^{-1})(r{+}s)[E_{12343},E_{234}]-(r{+}s)[E_{1234323},E_4]+(rs)^{-1}(r^2{-}s^2)E_{23}E_{123434},
\end{split}
\end{equation*}
that is
\begin{equation*}
[E_{12343},E_{234}]=(r^{-1}{-}r^{-2}s)E_{23}E_{123434}-r^{-1}s[E_{1234323},E_4].
\end{equation*}
So we get (6) \& (7).
%\begin{equation*}
% [E_{12343},E_{234}]=(r+s)^{-1}(E_{23}E_{123434}-s^2E_{123434}E_{23});
%\end{equation*}
%and
%\begin{equation*}
% [E_{1234323},E_4]=rs(r+s)^{-1}E_{12343423}.
%\end{equation*}

(8) follows directly from \ref{J}  and (7).

(9) follows directly from \ref{J},  Lemma \ref{l2} (1) and \ref{l4}.

(10):  Using \ref{J}, \ref{L}, Lemma \ref{l4} (7) and $[E_{1234}, E_{1234342}]=0$, we obtain
\begin{equation*}
\begin{split}
 &[E_{123432},E_{123434}]\\
&=(1{-}rs^{-1})[E_{123432}E_{1234},E_{34}]-rs[[E_{123432},E_{34}],E_{1234}]\\
&=r(s{-}r)[E_{123432},E_{34}]E_{1234}+(1{-}rs^{-1})E_{123432}E_{123434}-rs[[E_{123432},E_{34}],E_{1234}]\\
&=s^{-1}E_{1234}[E_{123432},E_{34}]-r^2[E_{123432},E_{34}]E_{1234}+(1{-}rs^{-1})E_{123432}E_{123434}\\
&=-s(r{+}s)^{-1}E_{1234}E_{12343423}+s(r{-}s)E_{1234}E_3E_{1234342}+(rs)^2(r{+}s)^{-1}E_{12343423}E_{1234}\\
 &\quad-(rs)^2(r{-}s)E_3E_{1234342}E_{1234}+(1{-}rs^{-1})E_{123432}E_{123434}\\
&=-s(r{+}s)^{-1}[E_{1234},E_{12343423}]+s(r{-}s)E_{12343}E_{1234342}+(1{-}rs^{-1})E_{123432}E_{123434}\\
&=-s(r{+}s)^{-1}(-r^2[E_{12343},E_{1234342}]+(r^2{-}s^2)E_{12343}E_{1234342})+s(r{-}s)E_{12343}E_{1234342}\\
 &\quad+(1{-}rs^{-1})E_{123432}E_{123434}\\
&=-s(r{+}s)^{-1}(-s^2[E_{12343},E_{1234342}]+(r^2{-}s^2)E_{1234342}E_{12343})\\
 &\quad +s(r{-}s)([E_{12343},E_{1234342}]+E_{1234342}E_{12343})+(1{-}rs^{-1})E_{123432}E_{123434}\\
&=r^2s(r{+}s)^{-1}(-s^{-2}[E_{123432},E_{123434}]+(s^{-2}{-}r^{-2})E_{123432}E_{123434})\\
 &\quad+(1{-}rs^{-1})E_{123432}E_{123434}\\
&=-r^2s^{-1}(r{+}s)^{-1}[E_{123432},E_{123434}].
\end{split}
\end{equation*}
%then
%$$(r^2+rs+s^2)[E_{123432},E_{123434}]=0.$$
We have $[E_{123432},E_{123434}]=0$, since $r^2+rs+s^2\neq0$.

(11) follows directly from (10) and $[E_{123434},E_3]=0$.

(12) follows directly from (9), $[E_{123434},E_3]=0$ and
\ref{J}.

(13) follows directly from $[E_{123434},E_4]=0$ and
$[E_2,E_4]=0$.

(14): Using \ref{J}, (12) and (10), we have
\begin{equation*}
\begin{split}
&[E_{12343},E_{1234342}]\\
&=-s^{-2}[E_{123432},E_{123434}]+(s^{-2}{-}r^{-2})E_{123432}E_{123434}\\
&=(s^{-2}{-}r^{-2})E_{123432}E_{123434}.
\end{split}
\end{equation*}

(15): Using  \ref{L}, (13) and (14), we have
\begin{equation*}
\begin{split}
&[E_{123434},E_{1234342}]\\
&=[E_{12343},[E_4,E_{1234342}]]{+}r^{-2}[[E_{12343},E_{1234342}],E_4]{+}(s^{-2}{-}r^{-2})[E_{12343},E_{1234342}]E_4\\
&=(1{-}r^2s^{-2})[E_{12343},E_4E_{1234342}]+r^{-2}[[E_{12343},E_{1234342}],E_4]\\
&+(s^{-2}{-}r^{-2})[E_{12343},E_{1234342}]E_4\\
&=(1{-}r^2s^{-2})E_{123434}E_{1234342}+s^{-2}[[E_{12343},E_{1234342}],E_4]\\
&=(1{-}r^2s^{-2})E_{123434}E_{1234342}+s^{-2}(s^{-2}{-}r^{-2})[E_{123432}E_{123434},E_4]\\
&=(1{-}r^2s^{-2})(E_{123434}E_{1234342}-rs^3E_{1234342}E_{123434})\\
&=(1{-}r^2s^{-2})[E_{123434},E_{1234342}].
\end{split}
\end{equation*}
%and then $r^2s^{-2}[E_{123434},E_{1234342}]=0$.
Therefore,
$[E_{123434},E_{1234342}]=0.$

(16) follows directly from \ref{J},  (7) and Lemma 3.4 (14).

(17) follows directly from  \ref{J}, Lemma \ref{l5} (1) and (8).
\end{proof}

\begin{lemm}\label{l6} The following relations hold in $U^{+}:$

\smallskip
$(1)\quad [E_{123434},E_{12343423}]=0;$

$(2)\quad [E_{1234342},E_{2}]=0;$

$(3)\quad [E_{1234342},E_{23}]=r^{-2}E_2E_{12343423}-r^2E_{12343423}E_2;$

$(4)\quad [E_{12343423},E_{2}]=0;$

$(5)\quad [E_{1234342},E_{12343423}]=0;$

$(6)\quad [E_{123434233},E_{3}]=0;$

$(7)\quad [E_{1234342},E_{123434233}]=r(r{-}s)E_{12343423}^2;$

$(8)\quad [E_{12343423},E_{123434233}]=0;$

$(9)\quad [E_{12343423},E_{23}]=r^{-2}s^{-2}E_2E_{123434233}-s^2E_{123434233}E_2;$

$(10)\quad
[E_{12343423},E_{233}]=r^{-1}s^{-2}(r{+}s)(E_{23}E_{123434233}-rs^3E_{123434233}E_{23});$

$(11)\quad [E_{123434233},E_{23}]=0;$

$(12)\quad [E_{1234342332},E_{3}]=0;$

$(13)\quad [E_{12343423},E_{1234342332}]=0;$

$(14)\quad [E_{123434233},E_{1234342332}]=0;$

$(15)\quad [E_{1234342332},E_2]=0;$

$(16)\quad [E_{1234342},E_{233}]=E_{23}E_{12343423}-rsE_{12343423}E_{23};$

$(17)\quad
[E_{1234323},E_{234}]=r(r{+}s)^{-1}(E_{23}E_{12343423}-s^2E_{12343423}E_{23});$

$(18)\quad
[E_{123432},E_{2343}]=r^2(r{-}s)E_{233}E_{1234342}+(r{-}s)E_{23}E_{12343423}$

\hskip4cm $+\,r^2sE_{234}E_{1234323}-sE_{1234323}E_{234}$.
\end{lemm}
\begin{proof}[{\bf Proof.}]
 (1): Using \ref{J}, Lemma \ref{l5} (15) and
$[E_{123434},E_3]=0$,  we obtain  the desired result.

(2) follows from \ref{J}  and Lemma \ref{l5} (1).
%: Since
%$E_{1234342}=[[E_{12343},E_4],E_2]=[[E_{12343},E_2],E_4]=[E_{123432},E_4],$
%we have
%\begin{equation*}
%\begin{split}
%[E_{1234342},E_{2}]
%&=[[E_{123432},E_4],E_{2}]\\
%&=[E_{123432},[E_4,E_{2}]]+[[E_{123432},E_2],E_{4}]\\
%&=0,
%\end{split}
%\end{equation*}
%where the third equality follows from \ref{l5} (1).

(3): Using (2) and Serre relations, we
obtain
\begin{equation*}
\begin{split}
[E_{1234342},E_{23}]
&=-s^2[E_{12343423},E_2]+(s^2{-}r^2)E_{12343423}E_2\\
&=r^{-2}E_2E_{12343423}-r^2E_{12343423}E_2.
\end{split}
\end{equation*}

(4): We have
\begin{equation*}
\begin{split}
&[E_{12343423},E_{2}]\\
&=[E_{123434},[E_{23},E_2]]+r^{-2}[E_{1234342},E_{23}]+(s^{-2}{-}r^{-2})E_{1234342}E_{23}\\
&=(1{-}r^2s^{-2})[E_{123434},E_{23}E_2]+r^{-2}[E_{1234342},E_{23}]+(s^{-2}{-}r^{-2})E_{1234342}E_{23}\\
&=(1{-}r^2s^{-2})E_{12343423}E_2+s^{-2}[E_{1234342},E_{23}]\\
&=(1{-}r^2s^{-2})E_{12343423}E_2+s^{-2}(r^{-2}E_2E_{12343423}-r^2E_{12343423}E_2)\\
&=r^2s^{-2}[E_{12343423},E_{2}].
\end{split}
\end{equation*}
So we get $[E_{12343423},E_{2}]=0$.

(5) follows directly from \ref{J},  (1) and  (4).

(6) follows directly from $[E_{123434},E_3]=0$,
$[E_{233},E_3]=0$ and \ref{J}.

(7): Using (5) and \ref{J}, we have
\begin{equation*}
\begin{split}
[E_{1234342},E_{123434233}]
&=-rs[E_{12343423},E_{12343423}]\\
&=rs(rs^{-1}{-}1)E_{12343423}^2\\
&=r(r{-}s)E_{12343423}^2.
\end{split}
\end{equation*}

(8): Using (6) and \ref{J}, we have
\begin{equation*}
\begin{split}
&[E_{12343423},E_{123434233}]\\&=[E_{1234342},[E_3,E_{123434233}]]+r^{-2}[[E_{1234342},E_{123434233}],E_3]\\
&\quad +(s^{-2}{-}r^{-2})[E_{1234342},E_{123434233}]E_3\\
&=(1{-}r^2s^{-2})[E_{1234342},E_3E_{123434233}]+r^{-2}[[E_{1234342},E_{123434233}],E_3]\\
&\quad +(s^{-2}{-}r^{-2})[E_{1234342},E_{123434233}]E_3\\
&=(1{-}r^2s^{-2})E_{12343423}E_{123434233}+s^{-2}[[E_{1234342},E_{123434233}],E_3]\\
&=(1{-}r^2s^{-2})E_{12343423}E_{123434233}+rs^{-2}(r{-}s)[E_{12343423}^2,E_3]\\
&=(1{-}r^2s^{-2})E_{12343423}E_{123434233}\\
&\quad +rs^{-2}(r{-}s)(rsE_{123434233}E_{12343423}+E_{12343423}E_{123434233})\\
&=(1{-}rs^{-1})(E_{12343423}E_{123434233}-r^2E_{123434233}E_{12343423})\\
&=(1{-}rs^{-1})[E_{12343423},E_{123434233}].
\end{split}
\end{equation*}
%and  $rs^{-1}[E_{12343423},E_{123434233}]=0.$
Therefore, $[E_{12343423},E_{123434233}]=0.$

(9): Using (4) and \ref{J}, we have
\begin{equation*}
\begin{split}
[E_{12343423},E_{23}]
&=-r^2[E_{123434233},E_2]+(r^2{-}s^2)E_{123434233}E_2\\
&=r^{-2}s^{-2}E_2E_{123434233}-s^2E_{123434233}E_2.
\end{split}
\end{equation*}

(10): Using \ref{J}, \ref{L}, (9) and (6),  we obtain
\begin{equation*}
\begin{split}
[E_{12343423},E_{233}]&
=[[E_{12343423},E_{23}],E_3]-rs[E_{123434233},E_{23}]\\
%&=r^{-2}s^{-2}[E_2E_{123434233},E_3]-s^2[E_{123434233}E_2,E_3]-rs[E_{123434233},E_{23}]\\
&=r^{-2}s^{-2}[E_2E_{123434233},E_3]-s^2[E_{123434233}E_2,E_3]\\
&\quad -rsE_{123434233}E_{23}+r^{-1}s^{-1}E_{23}E_{123434233}\\
&=s^{-1}(r^{-1}{+}s^{-1})E_{23}E_{123434233}-s(r{+}s)E_{123434233}E_{23}.
\end{split}
\end{equation*}

(11): Using (10) and \ref{J}, we have
\begin{equation*}
\begin{split}
&[E_{123434233},E_{23}]\\&=[E_{123434},[E_{233},E_{23}]]+r^{-2}[E_{12343423},E_{233}]+(s^{-2}{-}r^{-2})E_{12343423}E_{233}\\
&=(1{-}r^2s^{-2})[E_{123434},E_{233}E_{23}]+r^{-2}[E_{12343423},E_{233}]+(s^{-2}{-}r^{-2})E_{12343423}E_{233}\\
&=(1{-}r^2s^{-2})E_{123434233}E_{23}{+}s^{-2}[E_{12343423},E_{233}]\\
&=(1{-}r^2s^{-2})E_{123434233}E_{23}{+}s^{-3}(r^{-1}{+}s^{-1})E_{23}E_{123434233}{-}s^{-1}(r{+}s)E_{123434233}E_{23}\\
&=-(rs^{-1}{+}r^2s^{-2})(E_{123434233}E_{23}-r^{-2}s^{-2}E_{23}E_{123434233})\\
&=-(rs^{-1}{+}r^2s^{-2})[E_{123434233},E_{23}],
\end{split}
\end{equation*}
then $(1{+}rs^{-1}{+}r^2s^{-2})[E_{123434233},E_{23}]=0$. Therefore,
$[E_{123434233},E_{23}]=0$, since $1{+}rs^{-1}{+}r^2s^{-2}\neq0$.

(12): Using (6), (11) and \ref{J}, we have
$$[E_{1234342332},E_{3}]=[E_{123434233},E_{23}]=0.$$

(13) follows directly from  \ref{J}, (4)  and (8).

(14) follows directly from \ref{J}, (12) and (13).

(15): Using  \ref{J}   and (4), we
have
\begin{equation*}
\begin{split}
[E_{1234342332},E_2]&=[[E_{123434233},E_2],E_2]\\
&=[[E_{12343423},[E_3,E_2]],E_2]\\
&=[E_{12343423},[[E_3,E_2],E_2]]\\
&=0.
\end{split}
\end{equation*}

(16): On the one hand, we have

\begin{equation*}
\begin{split}
[E_{1234342},E_{233}]&=[E_{123434},[E_2,E_{233}]]+r^2s^2[E_{123434233},E_2]\\
&=r(r{-}s)[E_{123434},E_{23}^2]+r^2s^2E_{1234342332}\\
&=r(r{-}s)(E_{12343423}E_{23}+r^{-2}E_{23}E_{12343423})+r^2s^2E_{1234342332}.
\end{split}
\end{equation*}
On the other hand, using \ref{J}, (3) and \ref{L}, we have
\begin{equation*}
\begin{split}
[E_{1234342},E_{233}]&=[[E_{1234342},E_{23}],E_{3}]-rs[E_{12343423},E_{23}]\\
&=r^{-2}[E_2E_{12343423},E_{3}]-r^2[E_{12343423}E_2, E_3]-rs[E_{12343423},E_{23}]\\
&=(1{+}r^{-1}s)E_{23}E_{12343423}-(r^2{+}rs)E_{12343423}E_{23}\\
&\quad -r^2s^2E_{123434233}E_2+r^{-2}E_2E_{123434233}.
\end{split}
\end{equation*}
Comparing the expansions, we obtain
$$[E_{1234342},E_{233}]=E_{23}E_{12343423}-rsE_{12343423}E_{23}.$$

(17): Using \ref{J}, Lemma \ref{l4} (14) and Lemma \ref{l5} (4),  we have  $[E_{1234323}, E_{23}]=0.$

Using \ref{J}, Lemma \ref{l5} (7), we obtain
\begin{equation*}
\begin{split}
[E_{1234323},E_{234}]&=[[E_{1234323},E_{23}],E_{4}]+s^{-1}E_{23}[E_{1234323},E_4]-s[E_{1234323},E_4]E_{23}\\
&=r(r{+}s)^{-1}(E_{23}E_{12343423}-s^2E_{12343423}E_{23}).
\end{split}
\end{equation*}

(18) follows from  \ref{J}, (16) and Lemma \ref{l5} (17).
\end{proof}

Combining the above lemmas, we get the following
\begin{theorem}\label {t7}
$[E_{\b_i},E_{\b_{i+1}}]=0, \b_i\in \Phi^+, 1\leq i<24.$ $\hfill\Box$
%Furthermore, $[E_{\b_i},E_{\b_{j}}]=0$  if $ \b_i,\b_j\in  \Phi^+, \b_i+\b_j\notin \Phi^+$.
\end{theorem}

\begin{theorem} \label{t8}
For two positive roots  $\b_i<\b_j$, we have $[E_{\b_i},E_{\b_j}]\in B_{i,j}$, where $B_{i,j}$ is the subalgebra of $U^+$
generated by $\{E_{\b_k}\mid \b_i<\b_k<\b_j, \b_k\in \Phi^+\}$.
\end{theorem}
\begin{proof}[{\bf Proof.}]We can write $\b$ as a sum $\b_i+\b_j$ of two positive roots $\b_i<\b_j$.

Case 1.  $\b$ is a positive root. If $(\b_i,\b_j)$ is a minimal pair
or a Lyndon pair for $\b$, we have $[E_{\b_i}, E_{\b_j}]=E_{\b}\in
B_{i,j}$;  if $(\b_i,\b_j)$ is neither a minimal pair  nor a Lyndon
pair for $\b$, according to the root decomposition in Table I and
the relations given in the  following forms, we have $[E_{\b_i},
E_{\b_j}]\in B_{i,j}$.

\newcommand{\tabincell}[2]{\begin{tabular}{@{}#1@{}}#2\end{tabular}}

%\begin{table}[h]{}
%\end{table}
\bigskip
\begin{tiny}
\begin{center}
\begin{tabular}{|c|c|c|c|c|c|}

  \hline
   $h(\b)$  &\tabincell{c} { roots $\rightarrow$  root vectors}  & \tabincell{c} {relations}   \\
  \hline
  4 & \tabincell{c}{$\b_4\rightarrow E_{1233}$ \\$\b_6\rightarrow E_{1234}$\\$\b_{20}\rightarrow E_{2343}$} &
  \tabincell{c}{$(M),(L);$ \\$(M), (L),(L);$\\$(M), L3.2(3),(2); $}\\
  \hline
  5 & \tabincell{c}{$\b_5\rightarrow E_{12332}$ \\$\b_7\rightarrow E_{12343}$\\$\b_{21}\rightarrow E_{23434}$} &
  \tabincell{c}{$(M),L3.1 (3), (7);$ \\$(M), (L),L3.2 (4), (9);$\\$(M),(L);$}\\
  \hline
 6 & \tabincell{c}{$\b_9\rightarrow E_{123432}$ \\$\b_{11}\rightarrow E_{123434}$} &
 \tabincell{c}{$(M), L3.2 (10), (11),  L3.4 (1), (15);$\\$(M),(L),(L); $}\\
  \hline
 7 & \tabincell{c}{$\b_{10}\rightarrow E_{1234323}$ \\$\b_{12}\rightarrow E_{1234342}$} &  \tabincell{c}{$(M), (L), L3.2 (5), (12), (13);$\\$(M),(L),L3.2 (14); $}\\
  \hline
 8 & \tabincell{c}{$\b_{13}\rightarrow E_{12343423}$} &  \tabincell{c}{$(M), (L), L3.2 (15), L3.5 (6), (7), (8); $}\\
  \hline
 9 & \tabincell{c}{$\b_{14}\rightarrow E_{123434233}$} &  \tabincell{c}{$(M), (L), L3.5 (16), L3.4 (16), L3.3 (11);$}\\
  \hline
 10 & \tabincell{c}{$\b_{15}\rightarrow E_{1234342332}$} &  \tabincell{c}{$(M), (L), L3.6 (9), (16), (17), (18), [E_{12332},E_{23434}]\in B_{5,21};$}\\
  \hline
 11 & \tabincell{c}{$\b_{8}\rightarrow E_{12343123432}$} &  \tabincell{c}{$(M),  L3.4 (8), (18),(19), [E_{123},E_{12343423}]\in B_{3,13},$\\$ [E_{12},E_{123434233}]\in B_{2,14}, [E_{1},E_{1234342332}]\in B_{1,15}.$}\\
  \hline
 \end{tabular}
 \end{center}
 \end{tiny}
\bigskip

Case 2.  $\b$ is not a positive root. If $\b=\b_i+\b_j=2\b_k$, $\b_k$ is a positive root with
$\b_i<\b_k<\b_j$, then $[E_{\b_i}, E_{\b_j}]=AE_{\b_k}^2$, where $A$ is a non-zero coefficient.
This implies $[E_{\b_i}, E_{\b_j}]\in B_{i,j}$.  If $\b=\b_i+\b_j=\b_{i'}+\b_{j'}$
and $\b_i<\b_{i'}<\b<\b_{j'}<\b_j, \b_{i'},\b_{j'}\in \Phi^+$,
then $[E_{\b_i}, E_{\b_j}]$  is a linear  combination  of products
$E_{\b_{i'}}E_{\b_{j'}}$ and $E_{\b_{j'}}E_{\b_{i'}} $.
If $\b=\b_i+\b_j\neq \b_{i'}+\b_{j'}$, then $[E_{\b_i}, E_{\b_j}]=0$. So $[E_{\b_i}, E_{\b_j}]\in B_{i,j}$.

This completes the proof.
\end{proof}

%\medskip

\noindent {\it 3.3. Central elements.} Assume that $r$ is a
primitive $d$th root of unity, $s$ is a primitive
 $d'$th root of unity and $\ell$ is the least common multiple of $d$ and
 $d'$. From now on, we assume that $\Bbb{K}$ contains a primitive $\ell $th
 root of unity.

In the following lemmas, we adopt the following notational conventions:
$$[n]_t:=\frac{1-t^n}{1-t},\quad [n]_t!:=[n]_t[n-1]_t\cdots
[2]_t[1]_t,\quad
 \left[m\atop n\right]_t:=\frac{[m]_t!}{[n]_t![m-n]_t!}. $$
By convention $[0]_t=0$ and
$[0]_t!=1$.

The following lemmas are useful for deriving some commutation
relations.

\begin{lemm}\label{l9} Let $x, y,z$ be elements of associative algebra $\mathfrak{A}$ over
$\mathbb{K}$, $p, q, q_1,q_2\in\mathbb{K}$, for nonnegative integer $m$,
 then
the following assertions hold

$(1)$ $x^my=\sum\limits_{i=0}^{m}q^i\left[m\atop i\right]_p((\mbox{ad}_qx)_L^{(m-i)}y)x^i$, where $(\mbox{ad}_qx)_L^{(n)}y=x((\mbox{ad}_qx)_L^{(n-1)}y)-p^{n-1}q((\mbox{ad}_qx)_L^{(n-1)}y)x, n>0$, and set  $(\mbox{ad}_qx)_L^{(0)}y=y;$

$(2)$ $xy^m=\sum\limits_{i=0}^{m}q^i\left[m\atop i\right]_py^i((\mbox{ad}_qy)_R^{(m-i)}x),$
where
$(\mbox{ad}_qy)_R^{(n)}x=((\mbox{ad}_qy)_R^{(n-1)}x)y-p^{n-1}qy((\mbox{ad}_qy)_R^{(n-1)}x), n>0$,   and set  $(\mbox{ad}_qy)_R^{(0)}x=x;$

$(3)$ $(\mbox{ad}_{q_1q_2}x)_L^{(m)}(yz)=\sum_{i=0}^mq_1^i \left[m\atop i\right]_p ((\mbox{ad}_{q_1}x)_L^{(m-i)}y)((\mbox{ad}_{q_2}x)_L^{(i)}z).$
\end{lemm}
\begin{proof}[{\bf Proof.}] (1): We can prove it by induction on $m$. The
equality clearly holds for $m=0$, suppose that equality holds for
all $k\leq m$, when $k=m+1$, we have
\begin{equation*}
\begin{split}
x^{m+1}y&=x\sum\limits_{i=0}^mq^i\left[m\atop i\right]_p((\mbox{ad}_qx)_L^{(m-i)}y)x^i\\
       &=\sum\limits_{i=0}^mq^i\left[m\atop i\right]_p\left((\mbox{ad}_qx)_L^{(m+1-i)}y+p^{m-i}q(\mbox{ad}_qx)_L^{(m-i)}y)x\right)x^i\\
       &=\sum\limits_{i=0}^mq^i\left[m\atop i\right]_p((\mbox{ad}_qx)_L^{(m+1-i)}y)x^i+\sum\limits_{i=0}^mp^{m-i}q^{i+1}\left[m\atop i\right]_p((\mbox{ad}_qx)_L^{(m-i)}y)x^{i+1}\\
        \end{split}
\end{equation*}
\begin{equation*}
\begin{split}&=\sum\limits_{i=0}^mq^i\left[m\atop i\right]_p((\mbox{ad}_qx)_L^{(m+1-i)}y)x^i+\sum\limits_{i=1}^{m+1}p^{m-i+1}q^{i}\left[m\atop i{-}1\right]_p((\mbox{ad}_qx)_L^{(m-i+1)}y)x^{i}\\
      &=\sum\limits_{i=0}^{m+1}q^i\left(\left[m\atop i\right]_p+p^{m-i+1}\left[m\atop i{-}1\right]_p\right)((\mbox{ad}_qx)_L^{(m+1-i)}y)x^i\\
       &=\sum\limits_{i=0}^{m+1}q^i\left[m{+}1\atop i\right]_p((\mbox{ad}_qx)_L^{(m+1-i)}y)x^i,
\end{split}
\end{equation*}
which completes the proof.

(2) and (3) can be easily verified by induction on $m$.
\end{proof}

\begin{lemm}\label{l10}
$(1)$ There exists a positive integer $N$ such that for all $m\geq
N$,
$$(\mbox{ad}_q{E}_{\b_i})_L^{(m)} {E}_{\b_{j}}=0, \quad i<j$$
holds, where
$q=\langle\omega_{\b_{j}}^\prime,\omega_{\b_i}\rangle.$

$(2)$ There exists a positive integer $N$ such that for all $m\geq
N$,
 $$(\mbox{ad}_q{E}_{\b_j})_R^{(m)} {E}_{\b_i}=0, \quad i<j$$
holds, where
$q=\langle\omega_{\b_i}^\prime,\omega_{\b_j}\rangle.$
\end{lemm}
\begin{proof}[{\bf Proof.}]
(1)  holds for $j=i+1$ by Theorem \ref{t7}. Using
Theorem \ref{t8}, we get
$$[ {E}_{\b_i}, {E}_{\b_{i+k+1}}]\in \mbox{span}_{\mbox{alg}}\{\
 {E}_{\b_t}\mid i<t<i+k+1\},$$ using the induction hypothesis and Lemma \ref{l9} (3),  we know
(1) holds for $j=i+k+1$, which completes the proof.

The similar argument shows that (2) is true.
\end{proof}

\begin{prop}\label{p11}
%Assume that $\ell$ is big enough,
$E_{\b_i}^\ell$ $(1\leq
i\leq24)$ commutes with $E_j$ $(1\leq j\leq4)$.
\end{prop}
\begin{proof}[{\bf Proof.}] Set $E_j=E_{\b_j}$, for a simple root $\b_j$.
If $ {\b_i}<\b_j$, note that for $p\in\mathbb{K}$ with $p^\ell=1$ we have
$\left[\ell \atop k\right]_p=0$ except $k=0, \ell,$
moreover $q=\langle\omega_{\b_j}^\prime,\omega_{\alpha_i}\rangle^\ell=1.$
%Now assume that $\ell$ is big enough,
Using Lemmas \ref{l9} and \ref{l10}, we
get
$$ {E}_{\b_i}^\ell E_j=(\mbox{ad}_q {E}_{\b_i})_L^{(\ell)}E_j+E_j {E}_{\b_i}^\ell=E_j{E}_{\b_i}^\ell.$$
The same argument shows
that this holds for the case when $ \b_i< \b_j$. Thus we
complete the proof.
\end{proof}

\begin{lemm}\label{l12}
Let  ${\b_i}\in \Phi^+$ with $h(\b_i)>1$.  For  $1\leq j\leq4$,  the following relations hold:
$$[E_{\b_i}, F_j]_1 =\left\{
\begin{array}{llllll}
 \mathbf{E_{\b}}\omega_1', \,  \mathbf{E_{\b}}\in B_{8,22}, & i=8, j=1;\\
r^{-3}(r{-}s)E_{12343}^2\omega_2, & i=8, j=2;\\
r^{-3}s^{-1}(r^{2}{-}s^2)E_{12343}E_{12332}\omega_4, & i=8, j=4;\\
a\mathbf{E_{\b}}\omega_1', \   a\neq 0,\   \mathbf{E_{\b}}\in B_{15,22}, & j=1, 9\leqslant i\leqslant 15; \\
a\mathbf{E_{\b}}\omega_2', \  a\neq 0, \   \mathbf{E_{\b}}\in B_{21,24},  & j=2, 17\leqslant i\leqslant 21;\\
aE_{\b_{k}}\omega_j ,  &   if \  \b_i=\b_{k}{+}\a_j, \b_{k}<\b_i<\a_j, \b_k\in \Phi^+;\\
aE_{\b_{k}}\omega_j' ,  & if \    \b_i=\a_j{+}\b_{k}, \a_j<\b_i<\b_k, \b_k\in \Phi^+;\\
0,   &  \textit{other\  cases}.
\end{array}
\right.$$
%\begin{enumerate}
      %  \item[(1)]$[ {E}_{\b_8},F_1]_1=r^{-2}(E_{123432}E_{2343}-r^5sE_{2343}E_{123432})\omega_1';$
       % \item[(2)]$[ {E}_{\b_8},F_2]_1=r^{-2}(1-r^{-1}s)E_{12343}^2\omega_2;$ $[ {E}_{\b_8},F_4]_1=r^{-3}s^{-1}(r^{2}-s^2)E_{12343}E_{12332}\omega_4;$
       % \item[(3)]$[ {E}_{\b_i},F_1]_1=a\mathbf{E_{\b}}\omega_1'$, where $\mathbf{E_{\b}}\in B_{15,22}$, $a\neq 0$ for $9\leqslant i\leqslant 15;$
      %\item[(4)]$[ {E}_{\b_i},F_2]_1=a\mathbf{E_{\b}}\omega_2'$, where $\mathbf{E_{\b}}\in B_{21,24}$, $a\neq 0$ for $17\leqslant i\leqslant 21;$
% \item[(5)]$[ {E}_{\b_i},F_j]_1=aE_{\b_{k}}\omega_j$ if       $\b_i=\b_{k}+\a_j$, $\b_{k}<\b_i<\a_j$, $\b_k\in \Phi^+;$
    %  \item[(6)]$[ {E}_{\b_i},F_j]_1=aE_{\b_{k}}\omega_j'$ if       $\b_i=\a_j+\b_{k}$, $\a_j<\b_i<\b_k$, $\b_k\in \Phi^+;$
     % \item[(7)]$[ {E}_{\b_i},F_j]_1=0$   for other cases.
%\end{enumerate}
where $[x,y]_1=xy-yx$.
\end{lemm}
\begin{proof}[{\bf Proof.}]
The proof can be carried out for $j=1,2,3,4$:

(I) If $j=1$, for $i=8$, we calculate
\begin{gather*}
\begin{split}
[E_{12343},F_1]_1&=-E_{2343}\omega_1';\\
[E_{123432},F_1]_1&=(r^{-2}E_{2}E_{2343}-r^2E_{2343}E_2)\omega_1'.
\end{split}
\end{gather*}
So we obtain  $[E_{12343123432},F_1]_1=\mathbf{E_{\b}}\omega_1'$ by
using $$[E_{12343},E_{2343}]=(r{+}s)^{-1}[E_{233},E_{123434}],$$
where
\begin{equation*}
\begin{split}
\mathbf{E_{\b}}&=r^{-2}s^{-1}(r{+}s)(E_{123432}E_{2343}{-}r^4s^2E_{2343}E_{123432})\\
&\quad{+}\,r^{-4}(r{+}s)^{-1}(E_2[E_{233},E_{123434}]{-}
r^6[E_{233},E_{123434}]E_2)\in
B_{15,22}.
\end{split}
\end{equation*}

For $i\neq 8$,  by direct calculation, we obtain $[E_{\b_i}, F_1]_1=a\mathbf{E_{\b}}\omega_1'$, where $\mathbf{E_{\b}}\in B_{15,22}$, $a\neq 0$
for $i=5$ and $9\leqslant i\leqslant 15$;  $a=0$ for $16\leqslant i \leqslant 24$; for $i=2,3,4,6,7$,
$[ {E}_{\b_i},F_1]_1=aE_{\b_{k}}\omega_1'$, where $\b_i=\a_1+\b_{k}$, $\a_1<\b_i<\b_k$, $\b_k\in \Phi^+$.

(II) If $j=2$, for $i>16$, by direct calculation, we obtain $[E_{\b_i}, F_2]_1=a\mathbf{E_{\b}}\omega_2'$,
where $\mathbf{E_{\b}}\in B_{21,24}$, $a\neq 0$  for $17\leqslant i\leqslant 21$;  $a=0$ for $22\leqslant i \leqslant 24$.

 For $i=2$, we get $[E_{12},F_{2}]_1=r^{-2}E_1\omega_2$  by direct calculation, furthermore, we have
$[E_{123},F_2]_c=0$ and  $[{E}_{\b_i},F_2]_1=0$ for $i=3,4,6,7,11$.
Moreover, we have
$[{E}_{\b_5},F_2]_1=[E_{12332},F_2]_c=r^{-2}E_{1233}\omega_2$ and
$[{E}_{\b_9},F_2]_1=[E_{123432},F_2]_c=r^{-2}E_{12343}\omega_2.$ By
direct calculation,  we have
$$[E_{\b_8},E_2]_1=[E_{12343123432},F_2]_1=r^{-2}(1{-}r^{-1}s)E_{12343}^2\omega_2.$$
 Moreover, we have $[E_{\b_i}, F_2]_1=0$ for $i=10,13,14$; $[E_{\b_i},F_2]_1=aE_{\b_{k}}\omega_2 ,   \b_i=\b_{k}+\a_2, \b_{k}<\b_i<\a_2, \b_k\in \Phi^+,$ for $i=12,15$.

%$$[{E}_{\b_{10}},F_2]_1=[E_{1234323},F_2]_1=0,$$
%$$[{E}_{\b_{12}},F_2]_1=[E_{1234342},F_2]_1=r^{-2}E_{123434}\omega_2,$$
%$$[{E}_{\b_{13}},F_2]_1=[E_{12343423},F_2]_1=0,$$
%$$[E_{123434233},F_2]_1=0.$$ Moreover
%$$[{E}_{\b_{15}},F_2]_1=[E_{1234342332},F_2]_1=r^{-2}E_{123434233}\omega_2.$$

(III) If $j=3$, for $i=2, 5,6,8,9,11,12, 15,19,21$,  $[E_{\b_i}, F_3]=0$;  for $i=3,4,7,10,13,14,17,18,20$, we have $[ {E}_{\b_i},F_3]_1=aE_{\b_{i-1}}\omega_3$ if $\b_i=\b_{i-1}+\a_3$, $\b_{i-1}<\b_i<\a_3$; for $i=23$, we have
$[ {E}_{\b_{23}},F_3]_1=[ {E}_{34},F_3]_1=aE_{4}\omega_3'$.

%only one of
%conditions (1) and (2) holds; We get $[E_{1234},F_3]_c=0$ by direct
%calculation, so
%$$[E_{12343},F_3]_c=r^{-1}E_{1234}\omega_3,$$
%$$[E_{123432},F_3]_c=0,$$
%so  $[E_{12343123432},F_3]$ satisfies condition (1), Moreover
%$[E_{123434},F_3]_c=0,$ so $[E_{1234342},F_3]_c=0$, we have
%$$[E_{12343423},F_3]_c=r^{-2}(r+s)E_{1234342}\omega_3;$$
%$$[E_{123434233},F_3]_c=(r^{-1}+s^{-1})E_{12343423}\omega_3;$$
%$$[E_{1234342332},F_3]_c=0,$$
%by direct calculation, and $[E_{234},F_3]_c=0$, therefore
%$$[E_{2343},F_3]_c=r^{-1}E_{234}\omega_3,$$
%$$[E_{23434},F_3]_c=0.$$

(IV) If $j=4$, for $i=2, 3,4,5,10,14,15, 17,18$,  $[E_{\b_i}, F_4]=0$;  for $i=6,7,9,10,11,$
$12,13,19,20,21,23$, we have $[ {E}_{\b_i},F_4]_1=aE_{\b_{i-1}}\omega_4$ if
$\b_i=\b_{k}+\a_4$, $\b_{k}<\b_i<\a_4, \b_{k}\in \Phi^+$ ; for $i=8$, we have
\begin{gather*}
\begin{split}
[ E_{\b_8},F_4]_1&=r^{-3}s^{-1}(E_{1233}E_{123432}-E_{123432}E_{1233})\omega_4\\
&=r^{-3}s^{-1}(r^2{-}s^2)E_{12343}E_{12332}\omega_4.
\end{split}
\end{gather*}

%If $j=4$, for $2\leq i\leq6$ or $i=17, 18, 19, 23$, only one
%of the conditions (1) and (2) holds. Moreover, we have
%$$[E_{12343},F_4]_c=r^{-2}E_{1233}\omega_4,$$
%$$[E_{123432},F_4]_c=r^{-2}E_{12332}\omega_4,$$
%by direct calculation, so $[E_{12343123432},F_4]_c$ satisfies
%condition (1). Moreover, we have
%$$[E_{1234323},F_4]_c=0;$$
%$$[E_{123434},F_4]_c=(r^{-1}+s^{-1})E_{12343}\omega_4;$$
%$$[E_{1234342},F_4]_c=0;$$
%$$[E_{12343423},F_4]_c=r^{-2}s^{-1}(r+s)E_{1234323}\omega_4;$$
%$$[E_{123434233},F_4]_c=[E_{1234342332},F_4]_c=0.$$

We complete the proof.
\end{proof}

\begin{prop}\label{p13}
$ {E}_{\b_i}^\ell$ $
(1\leq i\leq24)$ commutes with $F_j$ $ (1\leq j\leq4)$.
\end{prop}
\begin{proof}[{\bf Proof.}]

 For $i \neq j$, we have $(\mbox{ad}_q{E}_{\b_i})_L^{(\ell)}F_j=0$, by Lemma \ref{l12};  using Lemma \ref{l9} for $q=1,
p=\langle\omega_{\b_i}^\prime,\omega_{\b_i}\rangle
\langle\omega_{\b_i}^\prime,\omega_{j}\rangle^{-1}$, we get:
$${E}_{\b_i}^\ell F_j=F_j{E}_{\b_i}^\ell+(\mbox{ad}_q {E}_{\b_i})_L^{(\ell)}F_j=F_j{E}_{\b_i}^\ell.$$

For $i=j$, we have
$$E_i^aF_i=F_iE_i^a+\left(\frac{r_i^a-s_i^a}{r_i-s_i}\right)E_i^{a-1}\frac{s_i^{-a+1}\omega_i-r_i^{-a+1}\omega_i^\prime}{r_i-s_i}, \quad 1\leq i\leq4,$$
 for $a=\ell$, we have  $E_i^\ell F_i=F_iE_i^\ell$,
which completes the proof.
\end{proof}

\begin{theorem}\label{t14}
%Assume that $\ell$ is big enough, then 
All $E_{\beta_i}^\ell$,
$F_{\beta_i}^\ell$ $ (1\leq i\leq24)$, and $\omega_k^\ell-1$,
$\omega_k^{\prime \ell}-1$ $ (1\leq k\leq4)$ are central in $U_{r,s}(F_4)$.
\end{theorem}
\begin{proof}[{\bf Proof.}]
%Since $\omega_k^\ell E_j\omega_k^{-\ell}=\langle
%\omega_j^\prime,\omega_k\rangle^\ell E_j$ and $r^\ell=s^\ell=1$, so $\langle
%\omega_j^\prime,\omega_k\rangle^\ell=1$ and $\omega_k^\ell-1$ commutes
%with $E_i (1\leq i\leq4)$, and similarly it commutes with $F_i (1\leq i\leq4)$, that is, $\omega_k^\ell-1 (1\leq k\leq4)$ are central.
%The same argument shows that $\omega_k^{\prime \ell}-1 (1\leq
%k\leq4)$ are central.
It follows from Propositions \ref{p11} and \ref{p13}, together with using $\tau$.
\end{proof}

\section{Restricted two-parameter quantum groups}

\noindent{\it 4.1.} In what follows, {\it we assume that $\ell$ is
odd.} Now we define the restricted two-parameter quantum group of type
$F_4$.
\begin{defi}\label{3.17}\  The
\textit{restricted two-parameter quantum group of type $F_4$}  is the quotient
$$\mathfrak{u}_{r,s}(F_4):=U_{r,s}(F_4)/\mathcal{I},$$
where $\mathcal {I}$ is the two-sided ideal of $U_{r,s}(F_4)$ generated
by  $E_{\b_i}^\ell, F_{\b_i}^\ell$ $ (1\leq i\leq24)$ and
$\omega_k^\ell-1, \,\omega_k^{\prime \ell}-1$ $ (1\leq k\leq4)$.
\end{defi}

The following elements form a basis of $\mathfrak{u}_{r,s}(F_4)$ by
Theorem \ref{BaseE} and Corollary \ref{BaseF}:
$$E_{\b_{24}}^{c_{24}}E_{\b_{23}}^{c_{23}}\cdots E_{\b_1}^{c_1}
\omega_{1}^{b_1}\omega_{2}^{b_2}\omega_{3}^{b_3}\omega_{4}^{b_4}\omega_{1}'^{b_1^\prime}\omega_{2}'^{b_2^\prime}\omega_{3}'^{b_3^\prime}\omega_{4}'^{b_4^\prime}
F_{\b_1}^{d_{1}}F_{\b_2}^{d_{2}}\cdots F_{\b_{24}}^{d_{24}},$$ where all powers are between $0$ and $\ell{-}1$, so we
have $\mbox{dim}\  \mathfrak{u}_{r,s}(F_4)=\ell^{56}.$

\medskip \noindent{\it 4.2.} The main theorem of this section is
\begin{theorem}\label{Hopf}
The ideal $\mathcal{I}$ is a Hopf ideal, and
$\mathfrak{u}_{r,s}(F_4)$ is a finite-dimensional
Hopf algebra.
\end{theorem}

To prove that $\mathfrak{u}_{r,s}(F_4)$ is a finite-dimensional Hopf
algebra,  it suffices to prove that $\mathcal {I}$ is a Hopf ideal.
We observe that $\varepsilon(\mathcal {I})=0$, and the co-multiplication
$\Delta$ (resp. antipode $S$) is a homomorphism (resp.
anti-homomorphism) of algebras, so it remains to show that
$\Delta(x)\in \mathcal{I}\otimes U+U\otimes \mathcal {I}$ and
$S(x)\in \mathcal {I}$ hold for each generator $x$. We first
calculate:
\begin{equation*}
\begin{split}
\Delta(\omega_k^\ell-1)&=\omega_k^\ell \otimes\omega_k^\ell-1\otimes1\\
                    &=(\omega_k^\ell-1)\otimes\omega_k^\ell+1\otimes(\omega_k^\ell-1)\\
                    &\in\mathcal{I}\otimes
U+U\otimes \mathcal {I};\\
S(\omega_k^\ell-1)&=-\omega_k^{-\ell}(\omega_k^\ell-1)\in\mathcal {I}.
\end{split}
\end{equation*}
The same argument shows that $\Delta(\omega_k'^\ell-1)\in \mathcal{I}\otimes
U+U\otimes \mathcal {I} $ and $S(\omega_k'^\ell-1)\in \mathcal{I}$.

 The proof of Theorem \ref{Hopf} will be done through the
following Proposition on the formulae of quantum root vectors under
the co-multiplication.

The following lemma is useful for deriving the formulae of non-simple
root vectors under the co-multiplication.

\begin{lemm}\label{LC} Let $X, Y, Z$ be elements of a  $\mathbb{K}$-algebra such that  $XY=\a
YX+Z$  for some $\a \in \mathbb{K}^*$ and $m$ a natural number. Then
the following assertions hold

$(1)$ If  $ZY=\b YZ$ for $\b(\neq \a)  \in \mathbb{K}$, then
$XY^m=\a^mY^mX+\frac{\a^m-\b^m}{\a-\b}Y^{m-1}Z$;

$(2)$ If  $XZ=\b ZX$ for $\b(\neq \a)  \in \mathbb{K}$, then
$X^mY=\a^mYX^m+\frac{\a^m{-}\b^m}{\a{-}\b}ZX^{m{-}1}$;

$(3)$ If  $ZY=\a^2 YZ$, $XZ=\a^2 ZX$, $\a^2\neq 1$, then
$$(X+Y)^m=\sum\limits_{m_1,m_2,m_3\in Z^+,\atop{ m_1+2m_2+m_3=l}}\frac{[m]_\a!}{[m_1]_\a![m_3]_\a![2m_2]_\a!!}Y^{m_1}Z^{m_2}X^{m_3},$$
where $[m]_\a=\frac{1{-}\a^m}{1{-}\a}$.
\end{lemm}
\begin{proof}[{\bf Proof.}]
(1) and  (2) follow from \cite{HW2}.

(3) For details, see the proof of Lemma 1 in \cite{B}.
\end{proof}

\begin{prop} For $\b_i\in \Phi^+$, we have
\begin{equation*}
\begin{split}
&(1) \quad \Delta(E_{\b_i}^\ell)\in \mathcal{I}\otimes U+U\otimes
\mathcal{I},\quad  S(E_{\b_i}^\ell)\in \mathcal{I};\\
&(2) \quad \Delta(F_{\b_i}^\ell)\in \mathcal{I}\otimes U+U\otimes
\mathcal{I},\quad  S(F_{\b_i}^\ell)\in \mathcal{I}.
\end{split}
\end{equation*}
\end{prop}
\begin{proof}[{\bf Proof.}]
(1) The proof $\Delta(E_{\b_i}^\ell)\in \mathcal{I}\otimes U+U\otimes
\mathcal{I}$ can be carried out in three steps.

Step 1.  Let $\b_i=\a_i$. Note that $\Delta(E_i)=E_i\otimes1+\omega_i\otimes
E_i$ and $(\omega_i\otimes E_i)(E_i\otimes1)=r_is_i^{-1}(E_i\otimes1)(\omega_i\otimes
E_i)$ for $1\leqslant i\leqslant 4$,  so
$$\Delta(E_i^n)=\sum_{j=0}^n\left[n\atop j\right]_{r_is_i^{-1}}E_i^{n-j}\omega_i^j\otimes E_i^j.$$
We have
$$\Delta(E_i^\ell)=E_i^\ell\otimes1+\omega_i^\ell \otimes E_i^\ell\in\mathcal{I}\otimes
U+U\otimes \mathcal {I}.$$

Step 2.  Let $\b_i=\a_j+\a_{j+1}$ for $1\leqslant j\leqslant 3$ and write $E_{\b_i}=E_{jj+1}$. Note that
$$\Delta(E_{jj+1})=E_{jj+1}\otimes1+\omega_{jj+1}\otimes E_{jj+1}+(1{-}a_{jj+1}a_{j+1j})E_j\omega_{j+1}\otimes E_{j+1},$$
and
\begin{equation*}
\begin{split}
&(\omega_{jj{+}1}\otimes E_{jj{+}1}+(1{-}a_{jj{+}1}a_{j{+}1j})E_j\omega_{j{+}1}\otimes E_j)(E_{jj{+}1}\otimes1)\\
&\
=a_{jj}a_{jj{+}1}a_{j{+}1j}a_{j{+}1j{+}1}(E_{jj{+}1}{\otimes}1)(\omega_{jj{+}1}{\otimes}
E_{jj{+}1}{+}(1{-}a_{jj{+}1}a_{j{+}1j})E_j\omega_{j{+}1}{\otimes}
E_{j{+}1}),
\end{split}
\end{equation*}
and $a_{jj}a_{jj+1}a_{j+1j}a_{j+1j+1}=r^2s^{-2}$ or $rs^{-1}$.
So
$$\Delta(E_{jj+1}^\ell)=E_{jj+1}^\ell\otimes1+\left(\omega_{jj+1}\otimes E_{jj+1}+(1{-}a_{jj+1}a_{j+1j})E_j\omega_{j+1}\otimes E_{j+1}\right)^\ell.$$

Now we calculate $\left(\omega_{jj+1}\otimes
E_{jj+1}+(1{-}a_{jj+1}a_{j+1j})E_j\omega_{j+1}\otimes
E_{j+1}\right)^\ell$, let $X=\omega_{jj+1}\otimes E_{jj+1},
Y=E_j\omega_{j+1}\otimes E_{j+1}.$

By a simple computation, we have
$XY=a_{jj}a_{jj+1}a_{j+1j}a_{j+1j+1}YX+Z_j$, where $Z_1=0,
Z_2=s^{-2}E_2\omega_{233}\otimes E_{233}, Z_3=0$. So
$\Delta(E_{12}^\ell),\Delta(E_{34}^\ell) \in \mathcal{I}\otimes
U+U\otimes \mathcal{I}$. We only need to consider $j=2$. In this case, we
have $XZ_2=r^2s^{-2}Z_2X, Z_2Y=r^2s^{-2}YZ_2$. Using Lemma \ref{LC}
(3), we know that each monomial in the expansion of $(X+Y)^\ell$ is of
the form
$$Y^{\ell_1}Z^{\ell_2}X^{l_3},\quad \ell_1+2\ell_2+\ell_3=\ell,$$
%since $l$ is odd, not all of $l_1$ and $l_3$ are 0,
and note that the coefficient of $Y^{\ell_1}Z^{\ell_2}X^{\ell_3}$ is
$\frac{[\ell]_{rs^{-1}}!}{[\ell_1]_{rs^{-1}}![\ell_3]_{rs^{-1}}![2\ell_2]_{rs^{-1}}!!}.$
So we have $(X+Y)^\ell=X^\ell+Y^\ell$ by assumptions
$[\ell]_{rs^{-1}}=0$ and $\ell$ being odd. Thus,
$\Delta(E_{23}^\ell)\in \mathcal {I}\otimes U+U\otimes \mathcal
{I}$.

Step 3. Similarly, the results hold for the other quantum root vectors.

In order to show that $S(E_{\b_i}^\ell)\in\mathcal{I}$, we use an
induction of height of positive root $\b_i$. It is trivial when height is $1$, but we have
$$\Delta(E_{\b_i}^\ell)=E_{\b_i}^\ell\otimes1+\omega_{\b_i}^\ell\otimes E_{\b_i}^\ell+(\ast),$$
where $(\ast)$ is a linear combination of the terms $A\otimes B$, $A$
and $B$ are monomials  of $\omega_k^\ell$ $ (1\leq k\leq4), E_{\b_j}^\ell
(\mbox{ht} \b_j<\mbox{ht} \b_i)$, by the induction
hypothesis, we have $S(E_{\b_j}^\ell)\in\mathcal {I}$, moreover,
$$0=\varepsilon(E_{\b_i}^\ell)1=\mu\circ(S\otimes\mbox{id})\circ\Delta(E_{\b_i}^\ell),$$
by the property of antipode, so we get
$S(E_{\b_i}^\ell)+\omega_{\b_i}^{-\ell}E_{\b_i}^\ell\in\mathcal{I}$,
i.e., $S(E_{\b_i}^\ell)\in\mathcal{I}$.

(2) Using $\tau$, we find that $\Delta(F_{\b_i}^\ell)\in \mathcal{I}\otimes U+U\otimes
\mathcal{I}, S(F_{\b_i}^\ell)\in \mathcal{I}$.
\end{proof}

\section{Isomorphisms  of $\mathfrak{u}_{r,s}(F_4)$}

%\smallskip

Write  $\mathfrak{u}=\mathfrak{u}_{r,s}=\mathfrak{u}_{r,s}(F_4)$.
Let $G$ denote the group generated by $\omega_i$, $\omega_i^\prime$
$( 1\leqslant i\leqslant 4)$.  We define subspaces
$\{\mathfrak{a}_k\}_{k=0}^\infty$ of $\mathfrak{u}$ as follows:
\begin{gather*}
\begin{split}\mathfrak{a}_0&=\mathbb{K}G, \quad \mathfrak{a}_1=\mathbb{K}G+\sum_{i=1}^4(\mathbb{K}E_iG+\mathbb{K}F_iG),\\
\mathfrak{a}_k&=(\mathfrak{a}_1)^k, \quad k\geq 1.
\end{split}\tag{5.1}
\end{gather*}

Note that $1\in\mathfrak{a}_0, \Delta(\mathfrak{a}_0)\subseteq
\mathfrak{a}_0\otimes\mathfrak{a}_0$, and $\mathfrak{a}_1$ generates
$\mathfrak{u}$ as algebra, moreover, we have
$\Delta(\mathfrak{a}_1)\subseteq
\mathfrak{a}_1\otimes\mathfrak{a}_0+\mathfrak{a}_0\otimes\mathfrak{a}_1$,
so $\{\mathfrak{a}_k\}_{k=0}^\infty$ is a coalgebra filtration of
$\mathfrak{u}$ and $\mathfrak{u}_0\subseteq\mathfrak{a}_0$ by
\cite{M}, where $\mathfrak{u}_0$ is the coradical of $\mathfrak{u}$,
which is the sum of all simple subcoalgebras of $\mathfrak{u}$. On
the other hand, any element of $G$ spans a simple subcoalgebra of
dimension $1$, so $\mathfrak{a}_0$ is the sum of some simple
subcoalgebras of
 $\mathfrak{u}$, hence $\mathfrak{a}_0\subseteq\mathfrak{u}_0$, therefore
 $\mathfrak{u}_0\subseteq\mathfrak{a}_0=\mathbb{K}G$.
That is to say, all simple subcoalgebras of $\mathfrak{u}$ are of
$1$-dimensional, so $\mathfrak{u}$ is a finite-dimensional pointed
Hopf algebra.

Let $\mathfrak{b}$ be the Hopf subalgebra of
$\mathfrak{u}=\mathfrak{u}_{r,s}(F_4)$ generated by $E_i$,
$\omega_i^{\pm1}$ $ (1\leqslant i\leqslant 4)$,
$\mathfrak{b}^\prime$ the Hopf subalgebra generated by $F_i$,
$(\omega_i^\prime)^{\pm1}$ $ (1\leqslant i\leqslant 4)$. The same
argument shows that $\mathfrak{b}$ and $\mathfrak{b}^\prime$ are
pointed Hopf algebras. We have
\begin{prop}\label{p51}
The restricted two-parameter quantum group $\mathfrak{u}_{r,s}(F_4)$
and its Hopf subalgebras $\mathfrak{b}$ and $\mathfrak{b}^\prime$
are finite-dimensional pointed Hopf algebras. \hfill\qed
 \end{prop}

 It follows from [{\bf M}, Lemma 5.5.1] that
$\mathfrak{a}_k\subseteq \mathfrak{u}_k$ for all $k$, where
$\{\mathfrak{u}_k\}$ is the \textit{coradical filtration} of
$\mathfrak{u}$ defined inductively by
$\mathfrak{u}_k=\Delta^{-1}(\mathfrak{u}\otimes
\mathfrak{u}_{k-1}+\mathfrak{u}_0\otimes \mathfrak{u}).$ In
particular, $\mathfrak{a}_1\subseteq \mathfrak{u}_1.$ By [{\bf M},
Theorem 5.4.1], as $\mathfrak{u}$ is pointed, $\mathfrak{u}_1$ is
spanned by the set of group-like elements $G$, together with all the
skew-primitive elements of $\mathfrak{u}$. We
have

\begin{lemm}\label{l52}
Assume that $rs^{-1}$ is $\ell$th primitive root of unity, then
$$\mathfrak{u}_1=\mathbb{K}G+\sum_{i=1}^4(\mathbb{K}E_iG+\mathbb{K}F_iG).$$
\end{lemm}

%\smallskip
 Given two group-like elements $g$ and $h$ in a
Hopf algebra $H$, let $P_{g,h}(H)$ denote the set of skew-primitive
elements of $H$ given by
$$
P_{g,h}(H)=\{\,x\in H\mid \Delta(x)=x\otimes g+h\otimes x\,\}.
$$

\begin{lemm}\label{l53}
Assume that $rs^{-1}$ is a primitive $\ell$th root of
unity. Then
%\begin{gather*}
%\begin{split}
\smallskip

$({\textrm{\rm i}})$ \;\
$P_{1,\omega_i}(\mathfrak{u}_{r,s})=\mathbb{K}(1-\omega_i)+\mathbb{K}E_i;$
$\quad
P_{1,\omega_i'^{-1}}(\mathfrak{u}_{r,s})=\mathbb{K}(1-\omega_i'^{-1})+\mathbb{K}F_i\omega_i'^{-1};$

$\qquad P_{1,\sigma}(\mathfrak{u}_{r,s})=\mathbb{K}(1-\sigma), \quad
\textrm{for}\
 \sigma\not\in\{\,\omega_i, \, \omega_i'^{-1}\mid 1\le i\le 4\,\}.$

\smallskip
$({\rm ii})$\;
$P_{\omega'_i,1}(\mathfrak{u}_{r,s})=\mathbb{K}(1-\omega'_i)+\mathbb{K}F_i;$
$\quad
P_{\omega_i^{-1},1}(\mathfrak{u}_{r,s})=\mathbb{K}(1-\omega_i^{-1})+\mathbb{K}E_i\omega_i^{-1};$

$\qquad P_{\sigma,1}(\mathfrak{u}_{r,s})=\mathbb{K}(1-\sigma), \quad
\textrm{for}\
 \sigma\not\in\{\,\omega_i^{-1}, \omega'_i\mid 1\le i\le 4\,\}.$
\end{lemm}
\begin{proof}[{\bf Proof.}]
Since $rs^{-1}$ is $\ell$-th root of unity, we get
$$\mathbb{K}G+\sum_{g,h\in G}P_{g,h}(\mathfrak{u})=\mathbb{K}G+\sum_{i=1}^4(\mathbb{K}E_iG+\mathbb{K}F_iG)$$
by Lemma \ref{l52}, that is, each $x\in P_{1,\sigma}(\mathfrak{u})$
is of the form
$$x=\sum_{g\in G}\gamma_gg+\sum_{g\in G}\sum_{i=1}^4\alpha_{i,g}E_ig+\sum_{g\in G}\sum_{i=1}^4\beta_{i,g}F_ig,\eqno{(*)}$$
where $\alpha_{i,g}, \beta_{i,g}, \gamma_g\in\mathbb{K}$. On the one
hand, applying the co-multiplication to the right hand of equality $(*)$, we get
that $\Delta(x)$ equals to
$$\sum_{g\in G}\gamma_gg\otimes g
+\sum_{g\in G}\sum_{i=1}^4\alpha_{i,g}(E_ig\otimes
g+\omega_ig\otimes E_ig)+\sum_{g\in
G}\sum_{i=1}^4\beta_{i,g}(g\otimes F_ig+F_ig\otimes\omega_i^\prime
g).$$ On the other hand, since $\Delta(x)=x\otimes1+\sigma\otimes
x$, $\Delta(x)$ equals to
$$\sum_{g\in G}\gamma_g(g\otimes1+\sigma\otimes g)+\sum_{g\in G}\sum_{i=1}^4\alpha_{i,g}(E_ig\otimes1+\sigma\otimes E_ig)
+\sum_{g\in G}\sum_{i=1}^4\beta_{i,g}(F_ig\otimes1+\sigma\otimes
F_ig).$$ By comparison, we get $\beta_{i,g}=0$ $
(g\neq\omega_i'^{-1})$, $\alpha_{i,g}=0$ $ (g\neq1)$,
$\gamma_\sigma=-\gamma_1$, $\gamma_g=0$ $ (g\neq1,\sigma)$.
Comparing $\Delta(x)$ with
$\Delta(\gamma_1(1-\sigma)+\sum\limits_{i=1}^4\alpha_{i,1}E_i+\sum\limits_{i=1}^4\beta_{1,\omega_i'^{-1}}F_i\omega_i'^{-1})$,
we get: if $\sigma{\notin}\{\omega_i\mid 1\leq i\leq4\}$, then
$\alpha_{i,1}=0$; if $\sigma{\notin}\{\omega_i'^{-1}\mid 1\leq
i\leq4\}$, then $\beta_{i,\omega_i'^{-1}}=0$.

If $\sigma=\omega_i$ for some $i$, then $\alpha_{j,1}=0$ $ (j\neq
i)$, $\beta_{j,\omega_j'^{-1}}=0$, hence
$x=\gamma_1(1{-}
\omega_i)+\alpha_{i,1}E_i;$ If
$\sigma=\omega_i'^{-1}$ for some $i$, then
$\beta_{j,\omega_j'^{-1}}=0$ $ (j\neq i)$, $\alpha_{j,1}=0$, hence
$x=\gamma_1(1-\omega_i'^{-1})+\beta_{i,\omega_i'^{-1}}F_i\omega_i'^{-1}.$
Therefore
\begin{equation*}
\begin{split}
P_{1,\omega_i}(\mathfrak{u})&=\mathbb{K}(1-\omega_i)+\mathbb{K}E_i,
\quad 1\leq
i\leq4,\\
P_{1,\omega_i'^{-1}}(\mathfrak{u})&=\mathbb{K}(1-\omega_i'^{-1})+\mathbb{K}F_i\omega_i'^{-1},\quad
1\leq
i\leq4,\\
P_{1,\sigma}(\mathfrak{u})&=\mathbb{K}(1-\sigma),\quad
\sigma{\notin}\{\,\omega_i,\omega_i'^{-1}\mid 1\leq i\leq4\,\}.
\end{split}
\end{equation*}

This completes the proof of  (i).

(ii) can be proved by the same argument.
\end{proof}

Now we prove the following  theorem on isomorphisms:

\begin{theorem}\label{t54}
Assume that $rs^{-1}$ and $r's'^{-1}$ are $\ell$th primitive roots
of unity and $\ell\neq3, 4$, $\zeta$ is the $2$-nd root of unity,
then there exists an isomorphism of Hopf algebras
$\varphi:\mathfrak{u}_{r,s}\rightarrow\mathfrak{u}_{r',s'}$ if and
only if either
%\begin{enumerate}

$(1)$ $(r',s')=\zeta(r,s):$

 $\varphi(\omega_i)=\tilde{\omega_i}$, $\varphi(\omega_i')=\widetilde{\omega_i}'$, $\varphi(E_i)=a_i\widetilde{E}_i$,
$\varphi(F_i)=\zeta^{\delta_{i,3}+\delta_{i,4}}a_i^{-1}\widetilde{F_i};$ or

$(2)$ $(r',s')=\zeta(s,r):$

$\varphi(\omega_i)=\widetilde{\omega_i}'^{-1}$, $\varphi(\omega_i')=\widetilde{\omega_i}^{-1}$,
$\varphi(E_i)=a_i\widetilde{F_i}\widetilde{\omega_i}'^{-1}$,
$\varphi(F_i)=\zeta^{\delta_{i,3}+\delta_{i,4}}a_i^{-1}\widetilde{\omega_i}^{-1}\widetilde{E_i}$, where $a_i\in K^*$.
%\end{enumerate}
\end{theorem}
\begin{proof}[{\bf Proof.}]
Assume that
$\varphi:\mathfrak{u}_{r,s}\rightarrow\mathfrak{u}_{r',s'}$ is a
Hopf algebra isomorphism  and $\widetilde{E}_i, \widetilde{F}_i,
\widetilde{\omega}_i,  \widetilde{\omega}_i'$ are generators of
$\mathfrak{u}_{r',s'}$. Since $\varphi$ is a morphism of coalgebras,
$$\Delta(\varphi(E_i))=(\varphi\otimes\varphi)\Delta(E_i)=\varphi(E_i)\otimes1+\varphi(\omega_i)\otimes\varphi(E_i),$$
and $\varphi(\omega_i)\in\mathbb{K}\widetilde{G}$, i.e.,
$\varphi(E_i)\in P_{1,\varphi(\omega_i)}(\mathfrak{u}_{r',s'})$.
Using Lemma \ref{l53} (i), we get
$\varphi(\omega_i)\in\{\widetilde{\omega}_j,\widetilde{\omega}_j'^{-1}\mid1\leq
j\leq4\}$. More precisely, either
$\varphi(\omega_i)=\widetilde{\omega}_j$ and
$\varphi(E_i)=a(1-\widetilde{\omega}_j)+b\widetilde{E}_j$, or
$\varphi(\omega_i)=\widetilde{\omega}_j'^{-1}$ and
$\varphi(E_i)=a(1-\widetilde{\omega}_j'^{-1})+b\widetilde{F}_j\widetilde{\omega}_j'^{-1}$
for some $j$ and  $a, b\in\mathbb{K}$. We have $a=0$ in all cases by
applying $\varphi$ to relations $\omega_iE_i=r_is_i^{-1}E_i\omega_i$
and assumption $r_i\neq s_i$. Thus $b\in K^*$.
%we get
%$$\widetilde{\omega_j}[a(1-\widetilde{\omega_j})+b\widetilde{E_j}]=r_is_i^{-1}[a(1-\widetilde{\omega_j})+b\widetilde{E_j}]\widetilde{\omega_j},$$
%or
%$$\widetilde{\omega_j}'^{-1}[a(1-\widetilde{\omega_j}'^{-1})+b\widetilde{F_j}\widetilde{\omega_j}'^{-1}]=r_is_i^{-1}[a(1-\widetilde{\omega_j}'^{-1})+b\widetilde{F_j}\widetilde{\omega_j}'^{-1}]\widetilde{\omega_j}'^{-1},$$
%since

We claim that the two cases cannot hold simultaneously. Indeed,
suppose $\varphi(\omega_i)=\widetilde{\omega}_j$, $
\varphi(\omega_{t})=\widetilde{\omega}_k'^{-1}$ with $t\neq i$, then
$\varphi(E_i)=a_j\widetilde{E}_j$,
$\varphi(E_{t})=a_k\widetilde{F}_k\widetilde{\omega}_k'^{-1}$ by the
above discussion.  Applying $\varphi$ to relations
$\omega_iE_i\omega_i^{-1}=\langle\omega_i',\omega_i\rangle E_i$,
$\omega_{t}E_{t}\omega_{t}^{-1}=\langle\omega_{t}',\omega_{t}\rangle
E_{t}$,  we get: $r_is_i^{-1}=r_j's_j'^{-1}$,
$r_{t}s_{t}^{-1}=r_k'^{-1}s_k'$. If $j=k$, we have one of three
cases:  $(rs^{-1})^2=1$,  $(rs^{-1})^3=1$,
 $(rs^{-1})^4=1$, each of which contradicts to the assumption.
 If $j\neq k$, we have
$(\widetilde{F}_k\widetilde{\omega}_k'^{-1})\widetilde{E}_j
=\langle\widetilde{\omega}_k',\widetilde{\omega}_j\rangle\widetilde{E}_j(\widetilde{F}_k\widetilde{\omega}_k'^{-1})$,
applying $\varphi^{-1}$, we get
$E_{t}E_i=\langle\widetilde{\omega}_k',\widetilde{\omega}_j\rangle
E_{i}E_{t}$, which contradicts to  $[E_i,E_{t}]\neq0$ with
$|i-t|=1$. We can use $\omega_{i+1}$ or $\omega_{i+2}$ substitute
$\omega_{i}$ when $|i-t|\geqslant 2$ and get the same result. So for
each $1\leq i\leq4$, only one of two cases holds: (i)
$\varphi(\omega_i)\in\{\widetilde{\omega}_j\mid1\leq j\leq4\}$, (ii)
$\varphi(\omega_i)\in\{\widetilde{\omega}_j'^{-1}\mid1\leq
j\leq4\}$.

(i) We claim that $\varphi(\omega_i)=\widetilde{\omega}_i$ $ (1\leq
i\leq4)$. Assume that $\varphi(\omega_2)=\widetilde{\omega}_j$,
$\varphi(\omega_3)=\widetilde{\omega}_k$, then
$\varphi(E_2)=a_j\widetilde{E}_j$, $\varphi(E_3)=a_k\widetilde{E}_k$.
If $|j-k|>1$, then
$\widetilde{E}_j\widetilde{E}_k-\widetilde{E}_k\widetilde{E}_j=0$,
but
$\varphi^{-1}(\widetilde{E}_j\widetilde{E}_k-\widetilde{E}_k\widetilde{E}_j)=(a_ja_k)^{-1}(E_2E_3-E_3E_2)\neq0$,
this is a contradiction, so $|j-k|=1$.
If $j$ and $k$ are connected
by one edge in Dynkin diagram,  we have $\varphi^{-1}((\mbox{ad}_l\widetilde{E}_{j+1})^{1-a_{j+1,j}}\widetilde{E}_j)\neq
0$ with $j<k$,  $\varphi^{-1}((\mbox{ad}_l\widetilde{E}_j)^{1-a_{j,j+1}}\widetilde{E}_{j+1})\neq 0$ with  $j>k$, which is a contradiction.
 So we have $(j,k)=(3,2)$ or $(2,3)$.  If
$\varphi(\omega_2)=\widetilde{\omega}_3$,
$\varphi(\omega_3)=\widetilde{\omega}_2$, applying $\varphi$ to
relations $\omega_2E_2=r^2s^{-2}E_2\omega_2$,
$\omega_3E_3=rs^{-1}E_3\omega_3$, we get $r's'^{-1}=r^2s^{-2}$,
$r'^2s'^{-2}=rs^{-1}$, so $(rs^{-1})^3=1$, which contradicts to the
assumption. So we get $\varphi(\omega_2)=\widetilde{\omega}_2$,
$\varphi(\omega_3)=\widetilde{\omega}_3$.  Assume that
$\varphi(\omega_1)=\widetilde{\omega}_{j_1}$, then $j_1$ and $2$ are
connected in Dynkin diagram, but $\varphi$ is a bijection, so
$\varphi(\omega_1)=\widetilde{\omega}_1$,
$\varphi(\omega_4)=\widetilde{\omega}_4$.

The same argument shows
that for $1\leq i\leq4$, $\varphi(\omega_i')=\widetilde{\omega}_i'$,
$\varphi(F_i)=b_i\widetilde{F}_i$. Applying $\varphi$ to relations
$\omega_iE_j=\langle\omega_j'$, $\omega_i\rangle E_j\omega_i$, $
i,j\in\{2,3\}$, we get $\langle\widetilde{\omega}_j'$,
$\widetilde{\omega}_i\rangle=\langle\omega_j'$, $\omega_i\rangle$,
or $r'^2s'^{-2}=r^2s^{-2}$, $r's'^{-1}=rs^{-1}$, $r'^{-2}=r^{-2}$,
$s'^2=s^2$, so $(r',s')=\zeta(r,s)$, applying $\varphi$ to relations
$[E_i,F_i]=\frac{\omega_i-\omega_i^\prime}{r_i-s_i}$, we get
$\varphi(F_i)=\zeta^{\delta_{i,3}+\delta_{i,4}}a_i^{-1}\widetilde{F}_i
$ $(1\leq i\leq4)$. It is clear that $\varphi$ preserves all
$(r,s)$-Serre relations, so $\varphi$ is an isomorphism of Hopf
algebras.

(ii) The same argument shows that
$\varphi(\omega_i)=\widetilde{\omega_i}'^{-1}$ $ (1\leq i\leq4)$ and
$\varphi(\omega_i')=\widetilde{\omega_i}^{-1}$,
$\varphi(E_i)=a_i\widetilde{F}_i\widetilde{\omega_i}'^{-1}$,
$\varphi(F_i)=b_i\widetilde{\omega_i}^{-1}\widetilde{E}_i$. Applying
$\varphi$ to relations $\omega_iE_j=\langle\omega_j',
\omega_i\rangle E_j\omega_i$, $i,\, j\in\{2,3\}$, we get
$\langle\widetilde{\omega_i}',
\widetilde{\omega}_j\rangle^{-1}=\langle\omega_j', \omega_i\rangle$,
or $r'^{-2}s'^2=r^2s^{-2}$, $r'^{-1}s'=rs^{-1}$, $r'^2=s^2$,
$s'^{-2}=r^{-2}$, so $(r',s')=\zeta(s,r)$, applying $\varphi$ to
relations $[E_i,F_i]=\frac{\omega_i-\omega_i^\prime}{r_i-s_i}$, we
get
$\varphi(F_i)=\zeta^{\delta_{i,3}+\delta_{i,4}}a_i^{-1}\widetilde{\omega_i}^{-1}\widetilde{E}_i$.

It is clear that $\varphi$ preserves all $(r,s)$-Serre relations, so
$\varphi$ is an isomorphism of Hopf algebras.
\end{proof}

\begin{remark} The above isomorphism Theorem 5.4 is important for our seeking for the possible new finite-dimensional pointed (graded) Hopf algebras of type $F_4$
especially when the order $\ell$ of root of unity is low, for instance, $\ell=5, 7$ that don't satisfy the assuming condition:
$(\ell, 210)=1$ as in \cite{AS2}, as has been witnessed by Benkart-Pereira-Witherspoon in their Example 5.6 \cite{BPW}, where
two-parameter quantum groups $\mathfrak u_{r,s}(\mathfrak{sl}_3)$'s with $\ell=4$ can be reduced to two non-isomorphic one-parameter ones
$\mathfrak u_{q,q^{-1}}(\mathfrak{sl}_3)$ and $\mathfrak u_{1,q}(\mathfrak{sl}_3)$ by terms of the Isomorphism Theorem 5.5 for type $A$ established by the last two authors in \cite{HW2}.
Based on the example observed by Benkart-Pereira-Witherspoon in type $A$, we dare to conjecture that there should exist the possible new finite-dimensional pointed Hopf algebras
of type $F_4$ under the assumption $\ell=5, 7$ that are simultaneously graded rather than filtered (via the lifting method proposed by Andruskiewitsch-Schneider in \cite{AS1}).
\end{remark}

\section{$\mathfrak{u}_{r,s}(F_4)$ is a Drinfel'd double}

 Let $\theta$ be a primitive $\ell$th root of
unity in $\mathbb{K}$, and write $r=\theta^y$, $s=\theta^z$.
\begin{lemm}\label{61}
Assume that $(4(y^4{+}z^4{-}y^2z^2),\ell)=1$ and $rs^{-1}$ is a primitive $\ell$th
 root of unity, then there exists a Hopf algebra isomorphism
$(\mathfrak{b}')^{coop}\cong\mathfrak{b}^*.$
\end{lemm}
\begin{proof}[{\bf Proof.}]
Define $\gamma_j\in \mathfrak{b}^*$ such that $\gamma_j$'s are
algebra homomorphisms with
$$
\gamma_{j}(E_i)=0, \, \textrm{ and } \ \gamma_{j}(\omega_i)=\langle \omega_j',\omega_i\rangle\quad 1\leqslant i, j\leqslant 4.
$$
So they are group-like elements in $\mathfrak{b}^*$. Define
$\eta_j=\sum\limits_{g \in G(\mathfrak b)}(E_jg)^*$, where
$G(\mathfrak{b})$ is  the group generated by $\omega_i\ (1\leq i\leq
4)$ and the asterisk denotes the dual basis element relative to the
PBW-basis of $\mathfrak{b}$. The isomorphism $\phi:
(\mathfrak{b}')^{coop}\rightarrow \mathfrak{b}^*$ is defined by
$$
\phi(\omega'_j)=\gamma_j, \qquad \phi(F_j)=\eta_j.
$$
First, we will check that $\phi$ is a Hopf algebra homomorphism and
then we will show that it is a bijection.

Clearly, the $\gamma_j$'s are invertible elements in
$\mathfrak{b}^*$ that commute with one another and
$\gamma_j^\ell=1$. Note that $\eta_j^\ell=0$, as it is $0$ on any
basis element of $\mathfrak{b}$. We calculate
$\gamma_j\eta_i\gamma_j^{-1}$: It is nonzero only on basis elements
of the form $E_i\omega_1^{k_1}\cdots\omega_4^{k_4}$, and on such an
element it takes the value
\begin{gather*}
\begin{split}
& (\gamma_j\otimes\eta_i\otimes \gamma_j^{-1})((E_i\otimes 1\otimes
1+ \omega_i\otimes E_i\otimes 1+\omega_i\otimes \omega_i\otimes
E_i)(\omega_1^{k_1}\cdots\omega_4^{k_4})^{\otimes 3}) \\&=
\gamma_j(\omega_i\omega_1^{k_1}\cdots\omega_4^{k_4})\eta_i(E_i\omega_1^{k_1}\cdots\omega_4^{k_4})
\gamma_j^{-1}(\omega_1^{k_1}\cdots\omega_4^{k_4})\\
&=\gamma_{j}(\omega_i).
\end{split}
\end{gather*}
So, we have $\gamma_j\eta_i\gamma_j^{-1}=\langle
\omega_j',\omega_i\rangle\,\eta_i$, which corresponds to relation
(F3) for $\mathfrak{b}'$.

Next, with the same argument in the proof of Lemma 5.1 in \cite{HW1} (see, pp. 263--264), we easily check that the same relations on $\eta_i$, $\eta_j$ as those
for $f_i$, $f_j$ in (F6). Hence, $\phi$ is an algebra homomorphism. %verify that $\phi$ preserves the Serre relations. We first prove

We have already seen that $\gamma_i$ is a group-like element in $\mathfrak b^*$. Now using the same argument as in the proof of Lemma 5.1 in \cite{HW1} (see, pp. 264--265), we have
%Next we prove that $\phi$ preserves coproduct, note that $\gamma_i$
%is a group-like element in $\mathfrak{b}^*$ by definition, and
%$$\Delta(\eta_i)(E_i\omega_1^{j_1}\cdots\omega_4^{j_4}\otimes\omega_1^{k_1}\cdots\omega_4^{k_4})=\eta_i(E_i\omega_1^{j_1+k_1}\cdots\omega_4^{j_4+k_4})=1,$$
%$$\Delta(\eta_i)(\omega_1^{j_1}\cdots\omega_4^{j_4}\otimes E_i\omega_1^{k_1}\cdots\omega_4^{k_4})=\eta_i(\omega_1^{j_1}\cdots\omega_4^{j_4}E_i\omega_1^{k_1}\cdots\omega_4^{k_4})=\langle\omega_i',\omega_1\rangle^{j_1}\ldots\langle\omega_i',\omega_4\rangle^{j_4}.$$
%$\Delta(\eta_i)$ takes value 0 on other PBW basis of
%$\mathfrak{b}\otimes\mathfrak{b}$, moreover, we have
%$$(\eta_i\otimes1+\gamma_i\otimes\eta_i)(E_i\omega_1^{j_1}\cdots\omega_4^{j_4}\otimes\omega_1^{k_1}\cdots\omega_4^{k_4})=1,$$
%$$(\eta_i\otimes1+\gamma_i\otimes\eta_i)(\omega_1^{j_1}\cdots\omega_4^{j_4}\otimes E_i\omega_1^{k_1}\cdots\omega_4^{k_4})=\langle\omega_i',\omega_1\rangle^{j_1}\ldots\langle\omega_i',\omega_4\rangle^{j_4}.$$
%$\eta_i\otimes1+\gamma_i\otimes\eta_i$ takes value $0$ on the other PBW
%basis elments of $\mathfrak{b}\otimes\mathfrak{b}$, so we get
$\Delta(\eta_i)=\eta_i\otimes1+\gamma_i\otimes\eta_i$, which means that $\phi$ is a Hopf algebra homomorphism.

Finally, we show that $\phi$ is a bijection. As $\dim\mathfrak{b}^*=\dim(\mathfrak{b}')^{coop}$, it suffices to
prove that $\phi$ is injective. By \cite{M}, we need only to prove that
$\phi|_{(\mathfrak{b}')_1^{coop}}$ is injective. Using
Lemma \ref{l53}, we get
$(\mathfrak{b}')_1^{coop}=\mathbb{K}G(\mathfrak{b}')+\sum\limits_{i=1}^4\mathbb{K}F_iG(\mathfrak{b}')$,
where $G(\mathfrak{b}')$ is the group generated by $\omega_i'$ $(1\leq i\leq4)$. We claim that
$$\mathrm{span}_{\mathbb{K}}\{\gamma_1^{k_1}\cdots\gamma_4^{k_4}\mid 0\leq k_i<\ell\}=\mathrm{span}_{\mathbb{K}}\{(\omega_1^{k_1}\cdots\omega_4^{k_4})^* \mid 0\leq k_i<\ell\}.$$
This is equivalent to saying that the $\gamma_1^{k_1}\cdots\gamma_4^{k_4}$'s span the space of characters over $\mathbb K$ of the finite group
$\mathbb{Z}/\ell\mathbb{Z}\times\mathbb{Z}/\ell\mathbb{Z}\times\mathbb{Z}/\ell\mathbb{Z}\times\mathbb{Z}/\ell\mathbb{Z}$
generated by $\omega_1,\cdots,\omega_4$. Note that all irreducible characters are
linearly independent in $(\mathbb{K}G(\mathfrak{b}))^*$. Since
$\mathbb{K}$ contains $\theta$ a primitive $\ell$th root of unity, any irreducible
characters are the functions $\chi_{i_1,i_2,i_3,i_4}$ given by $\chi_{i_1,i_2,i_3,i_4}(\omega_1^{k_1}\omega_2^{k_2}\omega_3^{k_3}\omega_4^{k_4})=
\theta^{i_1k_1+i_2k_2+i_3k_3+i_4k_4}$. Note that $\gamma_1=\chi_{2(y-z),-2y,0,0},
\gamma_2=\chi_{2z,2(y-z),-2y,0}, \gamma_3=\chi_{0,2z,y-z,-y},
\gamma_4=\chi_{0,0,z,y-z}$. Now we have to show that given $i_1,
i_2, i_3, i_4$, there exist $x_1, x_2, x_3, x_4$, such that
$\chi_{i_1, i_2, i_3,
i_4}=\gamma_1^{x_1}\gamma_2^{x_2}\gamma_3^{x_3}\gamma_4^{x_4}$, which
is equivalent to the existence of a solution to the linear equation $Ax=I$ in
$\mathbb{Z}/\ell\mathbb{Z}$, where
$$A=
\left(
\begin{array}{cccc}
2(y-z)&2z&0&0\\
-2y&2(y-z)&2z&0\\
0&-2y&y-z&z\\
0&0&-y&y-z
\end{array}
\right),
$$
$x$ is the transpose of $(x_1, x_2, x_3, x_4)$, $I$ is the transpose of
$(i_1, i_2, i_3, i_4)$. Note that $\mbox{det} A=4(y^4+z^4-y^2z^2)$,
$A$ is invertible when $(4(y^4+z^4-y^2z^2),l)=1$, so the equation
has a unique solution which proves our claim. In particular, this implies that the matrix
$$((\gamma_1^{k_1}\gamma_2^{k_2}\gamma_3^{k_3}\gamma_4^{k_4})(\omega_1^{j_1}\omega_2^{j_2}\omega_3^{j_3}\omega_4^{j_4}))_{\overline{k},\overline{j}}$$
is invertible, %where $\overline{k},\overline{j}\in\mathbb{Z}/l\mathbb{Z}$, so
and $\phi|_{\mathbb{K}G(\mathfrak{b}')}$ is bijective.

Now we shall show for each $i$ ($1\le i\le 4$) that the following matrix is invertible:
$$((\eta_i\gamma_1^{k_1}\gamma_2^{k_2}\gamma_3^{k_3}\gamma_4^{k_4})(E_i\omega_1^{j_1}\omega_2^{j_2}\omega_3^{j_3}\omega_4^{j_4}))_{\overline{k},\overline{j}}.$$
This will complete the proof that $\phi$ is injective on $(\mathfrak b')_1^{\text{coop}}$, as desired. In fact,
$$
\begin{array}{ll}
&(\eta_i\gamma_1^{k_1}\gamma_2^{k_2}\gamma_3^{k_3}\gamma_4^{k_4})(E_i\omega_1^{j_1}\omega_2^{j_2}\omega_3^{j_3}\omega_4^{j_4})\\
&\quad=(\eta_i\otimes\gamma_1^{k_1}\gamma_2^{k_2}\gamma_3^{k_3}\gamma_4^{k_4})(\Delta(E_i)\Delta(\omega_1^{j_1}\omega_2^{j_2}\omega_3^{j_3}\omega_4^{j_4}))\\
&\quad=(\eta_i\otimes\gamma_1^{k_1}\gamma_2^{k_2}\gamma_3^{k_3}\gamma_4^{k_4})(E_i\omega_1^{j_1}\omega_2^{j_2}\omega_3^{j_3}\omega_4^{j_4}\otimes\omega_1^{j_1}\omega_2^{j_2}\omega_3^{j_3}\omega_4^{j_4})\\
&\quad=(\gamma_1^{k_1}\gamma_2^{k_2}\gamma_3^{k_3}\gamma_4^{k_4})(\omega_1^{j_1}\omega_2^{j_2}\omega_3^{j_3}\omega_4^{j_4}),
\end{array}
$$
%so
%$$((\gamma_1^{k_1}\gamma_2^{k_2}\gamma_3^{k_3}\gamma_4^{k_4})(\omega_1^{j_1}\omega_2^{j_2}\omega_3^{j_3}\omega_4^{j_4}))_{\overline{k},\overline{j}}=
%((\eta_i\gamma_1^{k_1}\gamma_2^{k_2}\gamma_3^{k_3}\gamma_4^{k_4})(E_i\omega_1^{j_1}\omega_2^{j_2}\omega_3^{j_3}\omega_4^{j_4}))_{\overline{k},\overline{j}},$$
the matrix associated with $\phi|_{(\mathfrak{b}')_1^{coop}}$ is
block upper-triangular with the same block each and invertible. So
$\phi|_{(\mathfrak{b}')_1^{coop}}$ is injective.

This  completes the proof.
\end{proof}

With the same argument of Theorem 4.8 in
\cite{BW3}, we have

\begin{theorem}\label{62}
Assume that $(4(y^4{+}z^4{-}y^2z^2),\ell)=1$, then there exists a Hopf
algebra isomorphism $\mathcal
{D}(\mathfrak{b})\cong\mathfrak{u}_{r,s}(F_4).$\hfill\qed
\end{theorem}
%\begin{proof}[{\bf Proof.}]We denote $\check{E_i}=E_i\otimes1,
%\check{\omega_i}^{\pm1}=\omega_i^{\pm1}\otimes1,
%\check{\eta_i}=1\otimes\eta_i,
%\check{\gamma_i}^{\pm1}=1\otimes\gamma_i^{\pm1}.$ We define algebra
%homomorphism $\psi:(\mathfrak{b})\rightarrow\mathfrak{u}_{r,s}(F_4)$
%as follows:
%$$\psi(\check{E_i})=E_i, \psi(\check{\eta_i})=(s_i-r_i)F_i,$$
%$$\psi(\check{\omega_i}^{\pm1})=\omega_i^{\pm1}, \psi(\check{\gamma_i}^{\pm1})=\omega_i'^{\pm1}.$$
%The remaining proof is same as \cite{BW1}.
%\end{proof}

\section{Integrals }
 In this section, we determine
the left and right integrals in the Borel subalgebra $\mathfrak{b}$
of $\mathfrak{u}_{r,s}(F_4)$, which play an important role in knot invariant theory.
 Let $H$ be a
finite-dimensional Hopf algebra. An element $y\in H$ is a
\textit{left} (resp., \textit{right}) \textit{integral} if
$ay=\varepsilon(a)y$ (resp., $ya=\varepsilon(a)y$) for all $a \in
H$. The left (resp., right) integrals form a one-dimensional ideal
$\int^l_H$ (resp., $\int^r_H$) of $H$, and $S_H(\int^r_H)=\int^l_H$
under the antipode $S_H$ of $H$.

\smallskip
When $y\neq 0$ is a left integral of $H$, there exists a unique
group-like element $\gamma$ in the dual algebra $H^*$ (the so-called
\textit{distinguished group-like element} of $H^*$) such that
$ya=\gamma(a)y$.  If we had begun instead with a right integral
$y'\in H$, then we would have $ay'=\gamma^{-1}(a)y'$. This follows
from the fact that group-like elements are invertible, and the
following calculation, which can be found in [\textbf{M}, p.22]:
When $y' \in \int^r_H$, $S_H(y')$ is a nonzero multiple of $y$, so
that $S_H(y')S_H(a)=\gamma(S_H(a))S_H(y')$ for all $a \in H$.
Applying $S_H^{-1}$, we find $ay'=\gamma(S_H(a))y'$. As $\gamma$ is
group-like, we have
$\gamma(S_H(a))=S_{H^*}(\gamma)(a)=\gamma^{-1}(a).$

\smallskip
Now if $\lambda \neq 0$ is a right integral of $H^*$, then there
exists a unique group-like element $g$ of $H$ (the
\textit{distinguished group-like element} of $H$) such that
$\xi\lambda=\xi(g)\lambda$ for all $\xi \in H^*$. The algebra $H$ is
\textit{unimodular} (i.e., $\int^l_H=\int^r_H$) if and only if
$\gamma=\varepsilon$; and the dual algebra $H^*$ is unimodular if
and only if $g=1$.

\smallskip
The left and right $H^*$-module actions on $H$ are given by
$$
\xi\rightharpoonup a=\sum a_{(1)}\xi(a_{(2)}),\qquad
a\leftharpoonup\xi=\sum \xi(a_{(1)})a_{(2)}, \leqno(7.1)
$$
for all $\xi \in H^*$ and $a\in H$. In particular,
$\varepsilon\rightharpoonup a=a=a\leftharpoonup\varepsilon$ for all
$a \in H$.

\begin{theorem}\label{71}
Let $t=\prod\limits_{i=1}^4(1+\omega_i+\cdots+\omega_i^{\ell-1})$ and
$x=E_{\b_{24}}^{\ell-1}E_{\b_{23}}^{\ell-1}\cdots E_{\b_1}^{\ell-1},$
 then $y=tx$ $(resp.,\  y^\prime=xt )$ is a left $(resp.,\  right)$
integral in $\mathfrak{b}$.
\end{theorem}
\begin{proof}[{\bf Proof.}]
We have to show that $by=\varepsilon(b)y$
holds for all $b\in\mathfrak{b}$. It suffices to verify this for the generators $\omega_k$ and $E_k$
 since $\varepsilon$ is  an algebra homomorphism.

Observe that $\omega_kt=t=\vn(\om_k)t$ for all $k=1,2,3,4$, as
the $\om_i$'s commute and
$\om_k(1+\om_k+\cdots+\om_k^{\ell-1})=1+\om_k+\cdots+\om_k^{\ell-1}$.
From that, relation $\om_k y=\vn(\om_k)y$ is clear,  for  $1\leq k\leq4$.

Now we compute $E_ky$. We get
$$E_kt=\prod_{i=1}^4(1+\langle\omega_k^\prime,\omega_i\rangle^{-1}\omega_i+\cdots+\langle\omega_k^\prime,\omega_i\rangle^{-(\ell-1)}\omega_i^{\ell-1})E_k$$
from a direct computation. It remains to verify that
$E_kx=0=\varepsilon(E_k)x$.

We first prove a more general result: for $1\leq i\leq j\leq24$, let
$$\widehat{E_{j,i}}={E}_{\b_j}^{\ell-1}{E}_{\b_{j-1}}^{\ell-1}\cdots{E}_{\b_i}^{\ell-1},$$
we will prove ${E}_{\b_i}\widehat{{E}_{j,i}}=0$ by
induction on $j-i$.

If $j-i=0$, then
${E}_{\b_i}{E}_{\b_i}^{\ell-1}={E}_{\b_i}^\ell=0$. Assume that
the result holds for $j-i\leq n$, if $j-i=n+1$, using Lemma
\ref{l9} for
$p=\langle\omega_{\b_j}^\prime,\omega_{\b_j}\rangle,
q=\langle\omega_{\b_j}^\prime,\omega_{\b_i}\rangle$ and
Theorem \ref{t8}, we get
$${E}_{\b_i}{E}_{\b_j}^{\ell-1}=q^{\ell-1}{E}_{\b_j}^{\ell-1}{E}_{\b_i}
+\sum_{k=1}^{\ell-1}q^{k}\left[\ell{-}1 \atop k\right]_{p}{E}_{\b_j}^{k}\left((ad_{q}E_{\b_j})_{R}^{\ell-1-k}E_{\b_i}\right),$$
where $(ad_{q}E_{\b_j})_{R}^{\ell-1-k}E_{\b_i}\in B_{i,j}$.
 So
\begin{equation*}
\begin{split}
{E}_{\b_i}\widehat{{E}_{j,i}}&={E}_{\b_i}({E}_{\b_j}^{\ell-1}{E}_{\b_{j-1}}^{\ell-1}\cdots E_{\b_i}^{\ell-1})\\
&=q^{\ell-1}{E}_{\b_j}^{\ell-1}{E}_{\b_i}({E}_{\b_{j-1}}^{\ell-1}{E}_{\b_{j-2}}^{\ell-1}\cdots {E}_{\b_i}^{\ell-1})\\
&\quad + \sum_{k=1}^{\ell-1}q^{k}\left[\ell{-}1 \atop k\right]_{p}{E}_{\b_j}^{k}\left((ad_{q}E_{\b_j})_{R}^{\ell-1-k}E_{\b_i}\right)  ({E}_{\b_{j-1}}^{\ell-1}{E}_{\b_{j-2}}^{\ell-1}\cdots{E}_{\b_i}^{\ell-1}).
\end{split}
\end{equation*}
By the induction hypothesis, we have
$${E}_{\b_i}({E}_{\b_{j-1}}^{\ell-1}{E}_{\b_{j-2}}^{\ell-1}\cdots {E}_{\b_i}^{\ell-1})=0,$$
for every $E_{\b_t}\in B_{i,j}$,  $i<t<j$. By the
induction hypothesis, we have
$${E}_{\b_t}({E}_{\b_{j-1}}^{\ell-1}{E}_{\b_{j-2}}^{\ell-1}\cdots{E}_{\b_i}^{\ell-1})={E}_{\b_t}({E}_{\b_{j-1}}^{\ell-1}{E}_{\b_{j-2}}^{\ell-1}\cdots{E}_{\b_t}^{\ell-1}){E}_{\b_{t-1}}^{\ell-1}\cdots {E}_{\b_i}^{\ell-1}=0,$$
therefore, ${E}_{\b_i}\widehat{{E}_{j,i}}=0$, in
particular, $E_kx=0, 1\leq k\leq4$.

The proof for the right integral is similar.
%Similarly, for right integral, it suffices to verify that $xE_k=0$,
%we will prove $\widehat{{E}_{j,i}}{E}_{\b_j}=0$ by
%induction on $j-i$,  Assume that it holds for $j-i\leq n$, if
%$j-i=n+1$, using Lemma \ref{l9} for
%$p=\langle\omega_{\b_i}^\prime,\omega_{\b_i}\rangle,
%q=\langle\omega_{\b_j}^\prime,\omega_{\b_i}\rangle$  and
%Theorem \ref{t8}, we get
%$${E}_{\b_i}^{l-1}{E}_{\b_j}\in\langle\omega_{\b_j}^\prime,\omega_{\b_i}\rangle^{l-1}{E}_{\b_j}{E}_{\b_i}^{l-1}
%+\sum_{k=1}^{l-1}B_{i,j}{E}_{\b_i}^{l-1-k},$$ so
%\begin{equation*}
%\begin{split}
%\widehat{{E}_{j,i}}{E}_{\b_j}&=({E}_{\b_j}^{l-1}{E}_{\b_{j-1}}^{l-1}\cdots{E}_{\b_i}^{l-1}){E}_{\b_j}\\
%&\in\langle\omega_{\b_j}^\prime,\omega_{\b_i}\rangle^{l-1}({E}_{\b_j}^{l-1}{E}_{\b{j-2}}^{l-1}\cdots{E}_{\b{i+1}}^{l-1}){E}_{\b_j}{E}_{\b_i}^{l-1}\\
%&\quad +\sum\limits_{k=1}^{l-1}({E}_{\b_j}^{l-1}{E}_{\b_{j-2}}^{l-1}\cdots{E}_{\b_{i+1}}^{l-1})B_{i,j}{E}_{\b_i}^{l-1-k},
%\end{split}
%\end{equation*}
%using induction hypothesis, we have
%$$({E}_{\b_j}^{l-1}{E}_{\b_{j-2}}^{l-1}\cdots{E}_{\b_{i+1}}^{l-1}){E}_{\b_j}=0,$$
%assume ${E}_{\b_t}\in B_{i,j}$, using induction hypothesis, we
%have
%$$({E}_{\b_j}^{l-1}{E}_{\b_{j-2}}^{l-1}\cdots{E}_{\b_{i+1}}^{l-1}){E}_{\b_t}={E}_{\b_j}^{l-1}\cdots{E}_{\b_{t+1}}^{l-1}({E}_{\b_{t}}^{l-1}\cdots{E}_{\b_{i+1}}^{l-1}){E}_{\b_t}=0,$$
%therefore $\widehat{{E}_{j,i}}{E}_{\b_j}=0,$ in
%particular $xE_k=0, 1\leq k\leq4$.
\end{proof}

A finite-dimensional Hopf algebra $H$ is semisimple if and only if
$\varepsilon(\int_H^l)\neq0$ or $\varepsilon(\int_H^r)\neq0$. For
$H=\mathfrak{b}$, $y$ and $y'$ are respectively bases of
$\int_\mathfrak{b}^l$ and $\int_\mathfrak{b}^r$, we get
$\varepsilon(y)=0=\varepsilon(y^\prime)$ by Theorem \ref{71}, so we
get

\begin{coro}
$\mathfrak{b}$ is a finite-dimensional nonsemisimple pointed Hopf
algebra. $\hfill\Box$
\end{coro}

\section{To be a ribbon Hopf algebra}
A finite-dimensional Hopf algebra $H$ is
\textit{quasitriangular} if there is an invertible element $R=\sum
x_i\ot y_i$ in $H\ot H$ such that $\Delta^{op}(a)=R\Delta(a)R^{-1}$
for all $a \in H$, and $R$ satisfies the relations $(\Delta\ot
id)R=R_{13}R_{23}$, $(id\ot\Delta )R=R_{13}R_{12}$, where
$R_{12}=\sum x_i\ot y_i\ot 1$, $R_{13}=\sum x_i\ot 1\ot y_i$, and
$R_{23}=\sum 1\ot x_i\ot y_i$.

Write $u=\sum S(y_i)x_i$. Then $c=uS(u)$ is central in $H$ (cf.
\cite{Ka}) and is referred to as the \textit{Casimir element}.

An element $v\in H$ is a \textit{quasi-ribbon element} of
quasitriangular Hopf algebra $(H, R)$ if $\textrm{(i)}$\ $v^2=c$, \
$\textrm{(ii)}$\ $S(v)=v$, \ $\textrm{(iii)}$\ $\vn(v)=1$,\
$\textrm{(iv)}$\ $\Delta(v)=(R_{21}R)^{-1}(v\ot v)$, where
$R_{21}=\sum y_i\ot x_i$. If moreover $v$ is central in $H$, then
$v$ is a \textit{ribbon element}, and $(H, R, v)$ is called a
\textit{ribbon Hopf algebra}. Ribbon elements provide a very
effective means of constructing invariants of $3$-manifolds (cf.
\cite{RT}).

 The Drinfel'd double $D(A)$ of a
finite-dimensional Hopf algebra $A$ is quasitriangular, and
Kauffman-Radford (\cite{KR}) proved a criterion for $D(A)$ to be
ribbon.

\begin{theorem} {\rm([KR, Thm. 3])}\label{81}
Assume $A$ is a finite-dimensional Hopf algebra, let $g$ and
$\gamma$ be the distinguished group-like elements of $A$ and $A^*$
respectively. Then

$(\mathrm{i})$ \ $(D(A),R)$ has a quasi-ribbon element if and only
if there exist group-like elements $h\in A$, $\delta \in A^*$ such
that $h^2=g$, $\delta^2=\gamma$.

$(\mathrm{ii})$ \ $(D(A),R)$ has a ribbon element if and only if
there exist  $h$  and $\delta$  as in $(\mathrm{i})$ such that
$S^2(a)=h(\delta\rightharpoonup a\leftharpoonup \delta^{-1})h^{-1},
\ \forall\; a\in A$.
\end{theorem}

Next we determine the distinguished group-like elements in
$\mathfrak{b}$ and $\mathfrak{b}^*$, since all group-like elements
are precisely the algebra homomorphisms in
$\mbox{Alg}_{\mathbb{K}}(\mathfrak{b},\mathbb{K})$. To verify that
it is distinguished, it suffices to consider its values on
generators.

\begin{prop}
Let $2\rho=16\alpha_1+30\alpha_2+42\alpha_3+22\alpha_4$, where
$\rho$ is the half sum of positive roots of the root system $F_4$. Let $\gamma\in
\mbox{Alg}_{\mathbb{K}}(\mathfrak{b},\mathbb{K})$ be defined by
$$\gamma(E_k)=0,\quad \gamma(\omega_k)=\langle\omega_1',\omega_k\rangle^{16}\langle\omega_2',\omega_k\rangle^{30}\langle\omega_3',\omega_k\rangle^{42}\langle\omega_4',\omega_k\rangle^{22},$$
then $\gamma$ is a distinguished group-like element in
$\mathfrak{b}^*$.
\end{prop}
\begin{proof}[{\bf Proof.}] It suffices to verify $ya=\gamma(a)y$ for
$a=E_k$, $a=\omega_k$ $(1\leq k\leq4)$ and $y=tx$. It follows from the
proof of Theorem \ref{71} that $xE_k=0$ $(1\leq k\leq4)$, so
$yE_k=txE_k=0=\gamma(E_k)y$, and
\begin{equation*}
\begin{split}
y\omega_k&=tx\omega_k\\
&=t({E}_{\b_{24}}^{\ell-1}{E}_{\b_{23}}^{\ell-1}\cdots{E}_{\b_1}^{\ell-1})\omega_k\\
&=\langle\omega_1',\omega_k\rangle^{-16(\ell-1)}\langle\omega_2',\omega_k\rangle^{-30(\ell-1)}\langle\omega_3',\omega_k\rangle^{-42(\ell-1)}\langle\omega_4',\omega_k\rangle^{-22(\ell-1)}t\omega_kx\\
&=\langle\omega_1',\omega_k\rangle^{16}\langle\omega_2',\omega_k\rangle^{30}\langle\omega_3',\omega_k\rangle^{42}\langle\omega_4',\omega_k\rangle^{22}t\omega_kx\\
&=\gamma(\omega_k)y,
\end{split}
\end{equation*}
where $1\leq k\leq4$.
\end{proof}

With the assumption of Lemma \ref{61}, there exists a Hopf algebra
isomorphism $\phi:(\mathfrak{b}')^{coop}\rightarrow\mathfrak{b}^*$,
$\phi(\omega_j')=\gamma_j$, $ \phi(F_j)=\eta_j$, which induces a Hopf
pairing
$(\cdot|\cdot):\mathfrak{b}'\times\mathfrak{b}\rightarrow\mathbb{K}$,
where
$$(F_j|E_i)=\delta_{ij}, \quad (\omega_j'|\omega_i)=\langle\omega_j',\omega_i\rangle,$$ and
$(\cdot|\cdot)$ takes value $0$ on the other pairs of generators.  Write
$\omega_{2\rho}'=\omega_1'^{16}\omega_2'^{30}\omega_3'^{42}\omega_4'^{22}$,
then $(\omega_{2\rho}'|b)=\gamma(b)$ holds for all
$b\in\mathfrak{b}$, that is, $\gamma=(\omega_{2\rho}'|\cdot)$.

Note that there is a Hopf algebra isomorphism
$\mathfrak{b}_{s^{-1},r^{-1}}\cong(\mathfrak{b}')^{coop}\cong\mathfrak{b}^*$,
which is the composition $\phi\psi^{-1}$, where $\psi(F_i)=E_i$,
$\psi(\omega_i')=\omega_i$. Any left (resp., right) integral in
$\mathfrak{b}$ is mapped by $\phi\psi^{-1}$ to a left (resp., right)
integral in $\mathfrak{b}^*$, so we get

\begin{prop}\label{83}
Let $\lambda=\nu\eta, \lambda'=\eta\nu\in\mathfrak{b}^*$, where
$$\nu=\prod_{i=0}^4(1+\gamma_i+\cdots+\gamma_i^{\ell-1}),\,  \eta=\eta_{4}^{\ell-1}\eta_{34}^{\ell-1}\cdots\eta_{12}^{\ell-1}\eta_1^{\ell-1},$$
where for each good
 Lyndon word $w$, $\eta_w=[\eta_u,\eta_v]_{\langle\omega_u',\omega_v\rangle^{-1}}$, where $w=uv$ is
the Lyndon decomposition of $w$, then $\lambda, \lambda'$ are left and
right integrals in $\mathfrak{b}^*$. (Caution: $\eta_w$ here is
different from that in the proof of Theorem \ref{61}).$\hfill\Box$
\end{prop}

\begin{prop}\label{84}
$g=\omega_{2\rho}^{-1}$ is a distinguished group-like element in
$\mathfrak{b}$, and for all $1\leq i\leq4$, we have
$$(\omega_i'\mid g)=\langle\omega_i',\omega_1\rangle^{-16}\langle\omega_i',\omega_2\rangle^{-30}\langle\omega_i',\omega_3\rangle^{-42}\langle\omega_i',\omega_4\rangle^{-22}=\gamma_i(g).$$
\end{prop}
\begin{proof}[{\bf Proof.}] Let $F= {F}_{\b_{24}}^{\ell-1}
{F}_{\b_{23}}^{\ell-1}\cdots {F}_{\b_1}^{\ell-1}$, then
$$\omega_k'F=\langle\omega_i',\omega_1\rangle^{-16}\langle\omega_i',\omega_2\rangle^{-30}\langle\omega_i',\omega_3\rangle^{-42}\langle\omega_i',\omega_4\rangle^{-22}F\omega_k'.$$
Since
$\phi^{-1}(\lambda')=F\prod\limits_{i=1}^4(1+\omega_i'+\cdots+\omega_i'^{\ell-1})$,
we have
$$\gamma_k\lambda'=\langle\omega_i',\omega_1\rangle^{-16}\langle\omega_i',\omega_2\rangle^{-30}\langle\omega_i',\omega_3\rangle^{-42}\langle\omega_i',\omega_4\rangle^{-22}\lambda'$$
and $\eta_k\lambda'=0$. So we get $\xi\lambda'=\xi(g)\lambda'$ for
all $\xi\in\mathfrak{b}^*$.
\end{proof}

\begin{theorem}
Assume that $r$ and $s$ are $\ell$-th roots of unity, then $\mathcal
{D}(\mathfrak{b})$ has a ribbon element.
\end{theorem}
\begin{proof}[{\bf Proof.}] By Proposition \ref{84}, we know that
$g=\omega_{2\rho}^{-1}$ is a distinguished group-like element in
$\mathfrak{b}$. Let $h=\omega_\rho^{-1}\in\mathfrak{b}$, then $h$ is a
group-like element in $\mathfrak{b}$ and $h^2=g$. Set
$\delta=(\omega_{\rho}'|\cdot)$. As
$\gamma=(\omega_{2\rho}'|\cdot)$, $\delta$ is a group-like element in
$\mathfrak{b}$ and $\delta^2=\gamma$, where
$$\delta(E_k)=0, \quad \delta(\omega_k)=\langle\omega_1',\omega_k\rangle^8\langle\omega_2',\omega_k\rangle^{15}\langle\omega_3',\omega_k\rangle^{21}\langle\omega_4',\omega_k\rangle^{11},\quad 1\leq k \leq4.$$

Next we calculate:
$$h(\delta\rightharpoonup\omega_k\leftharpoonup\delta^{-1})h^{-1}=\delta(\omega_k)\delta^{-1}(\omega_k)h\omega_kh^{-1}=\omega_k=S^2(\omega_k),$$
\begin{equation*}
\begin{split}
h(\delta\rightharpoonup E_k\leftharpoonup\delta^{-1})h^{-1}&=\delta(1)\delta^{-1}(\omega_k)hE_kh^{-1}\\
&=\langle\omega_1',\omega_k\rangle^{-8}\langle\omega_2',\omega_k\rangle^{-15}\langle\omega_3',\omega_k\rangle^{-21}\langle\omega_4',\omega_k\rangle^{-11}hE_kh^{-1}\\
&=\langle\omega_{\rho}',\omega_k\rangle^{-1}\langle\omega_k',\omega_{\rho}\rangle^{-1}E_k\\
&=\omega_k^{-1}E_k\omega_k\\
&=S^2(E_k).
\end{split}
\end{equation*}
This completes the proof by Theorem \ref{81}.
\end{proof}

Under the assumption of Theorem \ref{62}, we have
$\mathfrak{u}_{r,s}(F_4)\cong\mathcal {D}(\mathfrak{b})$, so we get
\begin{coro}
Assume that $\theta$ is a primitive $\ell$-th root of unity in
$\mathbb{K}$, $r=\theta^y, s=\theta^z$, and
$(4(y^4{+}z^4{-}y^2z^2),\ell)=1$, then $\mathfrak{u}_{r,s}(F_4)$ is a
ribbon Hopf algebra. $\hfill\Box$
\end{coro}

\vskip30pt \centerline{\bf ACKNOWLEDGMENT}

\vskip15pt % Authors are indebted to the referees for their time and comments.
N.H. Hu would like to thank for the support of a
two-months invited guest professor fellowship from the University of
Paris Diderot during April 1st to May 31, 2016, as well as the hospitality of Professor Marc Rosso
and his colleagues. X.L. Wang was indebted to
host Professor Igor Pak for his kind invitation and help when she
visited University of California, Los Angeles, from Sept. of 2015 to August of 2016.
%\newpage

\bibliographystyle{amsalpha}

\end{document}